\documentclass[11pt,leqno]{amsart}

\usepackage[top=2.5 cm,bottom=2 cm,left=2.5 cm,right=2 cm]{geometry}
\usepackage[utf8]{inputenc}

\usepackage{bbm}
\usepackage{esint}
\usepackage{graphicx}
\usepackage{hyperref}
\usepackage{color}
\usepackage{amssymb}
\usepackage{amsfonts}
\usepackage{psfrag}
\newcounter{stepnb}

\usepackage{tikz}
\usetikzlibrary{shapes.geometric,calc,decorations.markings,math}

\tikzstyle{nodo}=[circle,draw,fill,inner sep=0pt,minimum size=%
1.5mm]
\tikzstyle{infinito}=[circle,inner sep=0pt,minimum size=0mm]

\newtheorem{theorem}{Theorem}[section]
\newtheorem{lemma}[theorem]{Lemma}
\newtheorem{proposition}[theorem]{Proposition}
\newtheorem{corol}[theorem]{Corollary}

\newtheorem{remark}[theorem]{Remark}

\newtheorem{definition}[theorem]{Definition}
\numberwithin{equation}{section}

\newcommand{\R}{\mathbb{R}}

\DeclareMathOperator*{\esslim}{ess\,lim}
\DeclareMathOperator*{\essinf}{ess\,inf}
\DeclareMathOperator*{\essup}{ess\,sup}

\newcommand{\ee}{\varepsilon}

\newcommand{\be}{\begin{equation}}
\newcommand{\eq}{\end{equation}}

\newcommand{\weaks}{\stackrel{*}{\rightharpoonup}}

\begin{document}

\title[Regularity and source-destination]{New regularity results for scalar conservation laws, and applications to a source-destination model for traffic flows on networks}

\author[S.~Dovetta]{Simone Dovetta}
\address{S.D. Dipartimento di Scienze di Base e Applicate per l'Ingegneria, Università degli Studi di Roma ``La Sapienza", via Antonio Scarpa 14, 00161 Roma, Italy.}
\email{simone.dovetta@uniroma1.it}
\author[E. Marconi]{Elio Marconi}
\address{E.M. EPFL B, Station 8, CH-1015 Lausanne, Switzerland.}
\email{elio.marconi@epfl.ch}
\author[L.~V.~Spinolo]{Laura V.~Spinolo}
\address{L.V.S. IMATI-CNR, via Ferrata 5, I-27100 Pavia, Italy.}
\email{spinolo@imati.cnr.it}
\maketitle
{
\rightskip .85 cm
\leftskip .85 cm
\parindent 0 pt
\begin{footnotesize}
We focus on entropy admissible solutions of scalar conservation laws in one space dimension and establish new regularity results with respect to time. First, we assume that the flux function $f$ is strictly convex and show that, for every $ x \in \R$, the total variation of the composite function $f \circ u(\cdot,  x)$ is controlled by the total variation of the initial datum. Next, we assume that $f$ is monotone and, under no convexity assumption, we show that, for every $x$, the total variation of the left and right trace $u(\cdot, x^\pm)$ is controlled by the total variation of the initial datum. We also exhibit a counter-example showing that in the first result the total variation bound does not extend to the function $u$, or equivalently that in the second result we cannot drop the monotonicity assumption. We then discuss applications to a source-destination model for traffic flows on road networks. We introduce a new approach, based on the analysis of transport equations with irregular coefficients, and, under the assumption that the network only contains so-called T-junctions, we establish existence and uniqueness results for merely bounded data in the class of solutions where the traffic is not congested. Our assumptions on the network and the traffic congestion are basically necessary to obtain well-posedness in view of a counter-example due to Bressan and Yu. We also 
establish stability and propagation of $BV$ regularity, and this is again interesting in view of recent counter-examples. 

\medskip\noindent
{\sc Keywords:} scalar conservation laws, regularity results, LWR model on networks, traffic models, source-destination model, multi-path approach 

\medskip\noindent
{\sc MSC (2010):  35L65}

\end{footnotesize}
}

\section{Introduction and main results}
We organize the introduction in two main parts: in the first one we discuss the regularity results, in the second one the applications to a traffic model. We conclude the introduction by providing the paper outline and recalling the main notation used in the paper. 
\subsection{Time regularity results for scalar conservation laws} 
We consider a scalar conservation law in one space dimension 
\be \label{e:cl}
    \partial_t u + \partial_x [f(u)]=0,
\eq
where $f \in C^2 (\R)$, 
and for the time being we focus on the Cauchy problem obtained by coupling~\eqref{e:cl} with the initial datum
\be \label{e:cpdatum}
     u(0, \cdot) = u_0.
\eq
The milestone paper by {Kru{\v{z}}kov~\cite{Kruzkov} establishes existence and uniqueness results for so-called \emph{entropy admissible solutions} of the Cauchy problem~\eqref{e:cl},\eqref{e:cpdatum}. It also establishes propagation of bounded total variation ($BV$) regularity: if $u_0 \in BV(\R)$, then the entropy admissible solution $u$ satisfies $u~\in~L^\infty (\R_+; BV(\R))$ and by using the equation this yields $u \in BV (]0, T[ \times \R)$ for every $T>0$. Since 
the pioneering work of Ole{\u\i}nik~\cite{Oleinik}, the investigation of the regularity properties of entropy admissible solutions has received considerable attention: here we only refer for an overview to the book by Dafermos~\cite{Dafermos:book}, to the recent contributions~\cite{AmbrosioDL,BianchiniMarconi,BGJ,CrippaOttoW,Jabin,Marconi} and to the references therein.

Our first regularity result establishes a uniform control on the total variation in time of the flux function $w:= f \circ u$ evaluated at any fixed  $x \in \R$. Despite the fact that the set $\{ (t, y): \; y=x \} \subseteq \R^2$ is negligible, the function $w(\cdot, x)$ is well defined owing to~\cite[Lemma 1.3.3]{Dafermos:book}, see also Lemma~\ref{l:dafermos} and Remark~\ref{r:pointwise} in the following. 
\begin{theorem}
\label{l:tvflux}
         Fix $f \in C^2 (\R)$  with  $f'' \leq 0$ or $f'' \ge 0$ and assume $u_0 \in BV (\R)$. Let  $u$ 
         be the entropy admissible solution of the Cauchy problem~\eqref{e:cl},~\eqref{e:cpdatum} and set $w : = f \circ u$, then
         \be 
         \label{e:tvflux2}
               \mathrm{TotVar} \, w (\cdot, x) \leq C \big( \mathrm{TotVar} \, u_0, \| f' \|_{L^\infty} \big) \quad \text{for every $x \in \R$}.
         \eq
        In the previous expression, we have set $w_0: = f \circ u_0$ and $ \| f' \|_{L^\infty} : = 
         \max_{ u \in [\operatorname{ess~inf} u_0, \operatorname{ess~sup} u_0] }| f'(u)|$.
\end{theorem}
Note that, in~\eqref{e:tvflux2}, $C ( \mathrm{TotVar} \, u_0,  \| f' \|_{L^\infty} )$ denotes a constant only depending on $\mathrm{TotVar} \, u_0$ and on the Lipschitz constant $ \| f' \|_{L^\infty}$ and its explicit expression can be reconstructed by following the proof of Theorem~\ref{l:tvflux}. In general, we cannot control the left hand side of~\eqref{e:tvflux2} with $ \mathrm{TotVar} \, w_0$, see Remark~\ref{r:iride} for a counterexample and some further considerations. 
  Note furthermore that in  \S\ref{ss:tvblowup} we exhibit a counter-example showing that, under the same assumptions as in Theorem~\ref{l:tvflux}, the total variation of $u(\cdot, x)$, or more precisely of the left and right traces $u(\cdot, x^\pm)$, can blow up in finite time. However, the next result shows that one can establish a uniform control on the total variation of the entropy admissible solution provided the function $f$ is monotone. 
\begin{proposition}\label{p:borraccia2}
Fix $f \in C^2(\R)$ and $u_0\in BV(\R)$. Assume moreover that $f'\ge 0$ or $f'\le 0$ on the interval $[\operatorname{ess~inf} u_0, \operatorname{ess~sup} u_0]$. Then the entropy solution of the Cauchy problem~\eqref{e:cl},\eqref{e:cpdatum} satisfies 
\be 
\label{e:borraccia2}
      \mathrm{TotVar} \, u (\cdot,  x^\pm) 
      \leq \mathrm{TotVar} \, u_0 , 
      \quad \text{for every $x \in \R$}.
\eq
In the above expression, $u (\cdot, x^\pm)$ denote the right and left trace of $u$ at $y= x$.
\end{proposition}
Note that, since  $u \in BV (]0, T[ \times \R)$ for every $T>0$, then the traces  $u (\cdot, x^\pm)$ are well defined owing to the general theory of $BV$ functions, see~\cite{AmbrosioFuscoPallara}. Note furthermore that in the statement of Proposition~\ref{p:borraccia2} we do not impose any concavity or convexity assumption on $f$. Also, the counterexample in  \S\ref{ss:tvblowup} shows that the monotonicity assumption in the statement of Proposition~\ref{p:borraccia2} cannot be dropped, even in the case of a convex flux. There are several possible extensions of Theorem~\ref{l:tvflux} and Proposition~\ref{p:borraccia2} to initial-boundary value problem: Corollary~\ref{c:stazione} provides the one we need in the proof of Theorem~\ref{t:propbvreg}. 
\subsection{Applications to a multi-path model for traffic flows on road networks}
The use of conservation laws in the macroscopic modeling of vehicular and pedestrian traffic started with the works by Lighthill, Whitham and Richards \cite{LW,R} and has since then flourished: we refer to~\cite{BellomoDogbe,BressanCanicGaravelloHertyPiccoli,GaravelloHanPiccoli} for an extended overview. In particular, since the paper by Holden and Risebro~\cite{HoldenRisebro}, several works have been devoted to the study of conservation laws model on road networks. In this framework, one of the main challenges is describing the behavior of the drivers at road junctions, see for instance the analysis and the discussion in~\cite{BressanNguyen,CocliteGaravelloPiccoli,GaravelloHanPiccoli,
HoldenRisebro}. 

In the present work we focus on the multi-path approach to a source-destination model for traffic flows on road networks. We refer to~\cite{BressanYu,GaravelloHanPiccoli,GaravelloPiccoli:CMS} for an extended discussion on source-destination models, but in a nutshell the very basic feature of these models is that drivers are divided in several \emph{populations} depending on the path they follow on the road network. On each road, the total car density is governed by a scalar conservation law as in the classical  
Lighthill, Whitham and Richards (LWR) model, whereas the rate of cars 
following a given path satisfies a transport equation where the coefficient 
depends on the solution of the conservation law. In~\cite{BressanYu,GaravelloPiccoli:CMS} the model was approached by relying on wave front-tracking techniques and hence one of the main points in the analyis 
  was the solution of the so-called Riemann problems at roads junctions, in the same spirit as in~\cite{CocliteGaravelloPiccoli,GaravelloPiccoli:AIHP}. In particular, in~\cite{GaravelloPiccoli:CMS} Garavello and Piccoli establish existence of a suitable notion of solution provided the data are a small $BV$ perturbation of an equilibrium and under further technical assumptions. In~\cite{BressanYu} Bressan and Yu, among other things, exhibit some counterexamples that we comment upon later. 

In the present work we focus on the same multi-path approach to the source-destination model as in~\cite{BrianiCristiani,HW}. In this approach, one focuses on \emph{paths} (each of them followed by a population of drivers) 
rather than on roads. Junctions apparently disappear or, more correctly, are hidden in the fact that the equation governing the evolution of the total car density is discontinuous at each junction. In~\cite{BrianiCristiani} Briani and Cristiani regard the multi-path model as a system of conservation laws with discontinuous fluxes and discuss the theoretical properties of a related Godunov-type numerical scheme. In {\color{black}the present} work we approach the multi-path model by relying on the theory of transport equations with low regularity coefficients. This allows us to provide a simple and neat formulation of the problem and in particular of the boundary conditions in a very weak $L^\infty$ framework. By relying on results obtained in the companion paper~\cite{DovettaMarconiSpinolo}, we establish existence and uniqueness results for $L^\infty$ data under the assumptions that the network only contains T-junctions, that is junctions with only one incoming road, and that the traffic is not congested.  These assumptions are obviously 
restrictive, but basically necessary to obtain well-posedness in view of a counterexample due to Bressan and Yu~\cite{BressanYu}. More precisely,~\cite[Example 3]{BressanYu} involves a simple network consisting of two incoming and two outgoing roads where the source-destination model has two distinct solutions, one where the traffic is congested and one where it is not. This shows that uniquenesss can be violated if we do not require the condition that the traffic is not congested. On the other hand, networks only containing T-junctions are the only ones where one can reasonably hope for propagation of the condition that the traffic is not congested. {\color{black}See also~\cite{GaravelloMarcellini} for another recent work where the authors restrict to T-junctions.} In the present work we also establish propagation of $BV$ regularity and stability, and these results are again interesting in view of counterexamples in~\cite{BressanYu} that we discuss in the following. 

We now provide the detailed description of the multi-path approach. To simplify the exposition, we directly focus on the case of a network only containing T-junctions like the one in Figure~\ref{f:rete}, but this introductory part and Definition~\ref{d:adsol} extend to more general networks. Our network consists of a collection of $h$ \emph{roads} $I_i, \dots, I_h$, each of them parameterized by a bounded interval\footnote{Our analysis straightforwardly extends to the case where the roads can have infinite length.} and running from a junction point (or from the source) to another (or to a destination). We also work with the $m$ \emph{paths} $P_1, \dots, P_m$: each of them is a collection of consecutive roads starting from the source and ending in a destination. We fix a time 
interval $[0, T]$ and  for every $i=1, \dots, h$, we denote by $\rho_i$ the total car density on the road $I_i$ and  as in the classical LWR model we assume that $\rho_i$ is an entropy admissible solution of the conservation law 
\be
\label{e:foglio}
    \partial_t \rho_i + \partial_x [v(\rho_i) \rho_i] =0    
     \quad \text{
on $ ]0, T[ \times I_i$}.
\eq
In the previous expression, the velocity function $v$ satisfies    
\be \label{e:v}
     v \in C^2 (\R ), \quad v (\rho^\mathrm{max})=0, \quad v \ge 0 \; \text{on $[0, \rho^\mathrm{max}]$}. 
     \eq 
Here the constant $\rho^\mathrm{max}>0$ denotes the maximum possible car density, corresponding to bumper-to-bumper packing.  The flux function satisfies 
\be \label{e:f}
    g(z) : = v(z) z,  \quad g' >0 \; \text{on $]0, \rho^\ast[$}, \quad 
   g' \leq 0 \; \text{on $]\rho^\ast, \rho^\mathrm{max}[$},
\eq
where the density $\rho^\ast<  \rho^\mathrm{max}$ denotes the transition between free and congested traffic. We remark in passing that we are \emph{not} making the assumption that $g$ is concave.  We denote by $\theta_1, \dots, \theta_m$ the traffic-type functions, that is for every $k=1, \dots, m$ the function $\theta_k$ represents the fraction of cars following the path $P_k$. It is governed by the equation  
\be
\label{e:te}
    \partial_t [r_k \theta_k ]+ \partial_x [v(r_k) r_k \theta_k] =0    \quad \text{
on $ ]0, T[ \times P_k$},
\eq 
where 
\begin{equation} \label{e:patches}
r_k = \rho_i,   \quad \text{a.e. on $]0, T[ \times I_i$ for every $i$ such that $I_i \subseteq P_k$,}
\end{equation}
that is $r_k$ is obtained by patching together the $\rho_i$-s. Note that by combining~\eqref{e:foglio},~\eqref{e:te} and~\eqref{e:patches} we formally obtain
$$
   \partial_t \theta_k + v (r_k) \partial_x \theta_k =0,
$$
that is a transport equation. Note, however, that, in view of the general theory of conservation laws~\cite{Dafermos:book}, the best regularity one can hope for is $r_k \in BV (]0, T[ \times P_k)$ and in this framework the product $v (r_k) \partial_x \theta_k$ is highly ill defined since in general $\partial_x \theta_k$ is only a distribution. 
For every $i=1, \dots, h$ and $k=1, \dots, m$, we fix $\rho_{i0} \in L^\infty (I_i)$ and 
$\theta_{k0} \in L^\infty (P_k)$ and we augment~\eqref{e:foglio} and~\eqref{e:te} with the initial conditions 
\be \label{e:id1}
     \rho_i(0, \cdot) = \rho_{i0}, \quad  0 \leq \rho_{i0} \leq \rho^{\mathrm{max}} \quad \text{a.e. on $I_i$}
\eq
and 
\be \label{e:id2}
      \theta_k (0, \cdot) =
     \theta_{k0} \quad \text{a.e. on $P_k$}. 
\eq 	
Since we are focusing on a network only containing T-junctions, all the paths have the same origin $a$ and start with the same road $I_1$. We fix $\bar \rho \in L^\infty (]0, T[)$ and impose 
\be 
\label{e:bd1}
    \rho_1 (\cdot, a) =
    \bar \rho,  \quad  0 \leq \bar \rho \leq \rho^{\mathrm{max}}, \quad \text{a.e. on $]0, T[$.} 
\eq
The above datum is attained in the sense of Bardos, LeRoux and N\'ed\'elec~\cite{BLN79}, see the discussion in \S\ref{ss:BLN}. 
For every $k=1, \dots, m,$ we fix $\bar \theta_k \in L^\infty (\R)$ and we impose the boundary condition
\be 
\label{e:bd2}
      \theta_k (\cdot, a)    =
        \bar \theta_k.
\eq
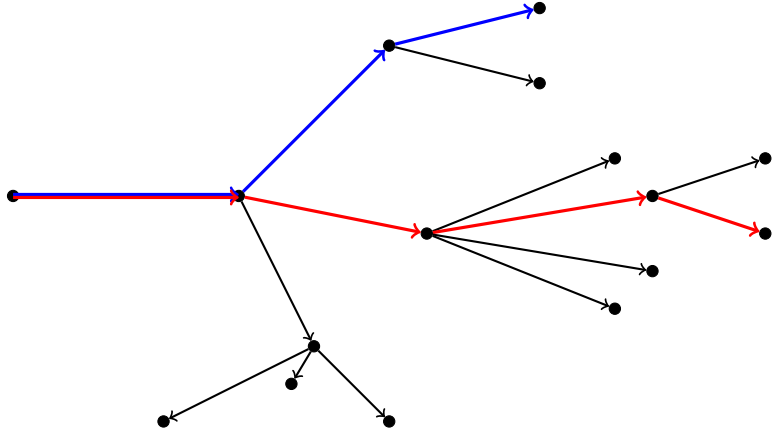
\begin{figure}
	\centering
	\begin{tikzpicture}
	\node at (-1,0) [nodo] (S) {};
	\node at (2,0) [nodo] (N1) {};
	\node at (4,2) [nodo] (N21) {};
	\node at (3,-2) [nodo] (N23) {};
	\node at (4.5,-.5) [nodo] (N22) {};
	\node at (6,2.5) [nodo] (D1) {};
	\node at (6,1.5) [nodo] (D2) {};
	\node at (7,.5) [nodo] (D3) {};
	\node at (7.5,0) [nodo] (D4) {};
	\node at (7.5,-1) [nodo] (D5) {};
	\node at (7,-1.5) [nodo] (D6) {};
	\node at (1,-3) [nodo] (D7) {};
	\node at (2.7,-2.5) [nodo] (D8) {};
	\node at (4,-3) [nodo] (D9) {};
	\node at (9,.5) [nodo] (D41) {};
	\node at (9,-.5) [nodo] (D42) {};
	
	\draw[->,very thick, blue] (N1)--(N21);
	\draw[->,very thick,red] (N1)--(N22);
	\draw[->,thick] (N1)--(N23);
	\draw[->,very thick, blue] (N21)--(D1);
	\draw[->,thick] (N21)--(D2);
	\draw[->,thick] (N22)--(D3);
	\draw[->,very thick, red] (N22)--(D4);
	\draw[->,thick] (N22)--(D5);
	\draw[->,thick] (N22)--(D6);
	\draw[->,thick] (N23)--(D7);
	\draw[->,thick] (N23)--(D8);
	\draw[->,thick] (N23)--(D9);
	\draw[->,thick] (D4)--(D41);
	
    \draw[->,very thick,red] (D4)--(D42);
	
	\draw[->, very  thick, blue] (-1,.02)--(2,.02);
	\draw[->,very thick, red] (-1,-.02)--(2,-.02);
	\end{tikzpicture}
	\caption{Example of a road network involving T-junctions only. The network has 15 roads, 1 source, 10 destinations and 10 paths {\color{black}(two of them highlighted in different colors).}}
	\label{f:rete}
\end{figure}
The above datum is attained in the sense of the \emph{distributional traces} as in~\cite{CrippaDonadelloSpinolo}, see Definition~\ref{d:te}. 
To conclude, we recall that $\theta_k$ represents the fraction of cars following the path $P_k$ and hence the physical range is 
\be \label{e:sommaunoid}
    0\leq  \theta_{k0} \leq 1  \; \text{a.e. on $P_k, \; \forall k=1, \dots, m$, } \qquad       \sum_{k: I_i \subseteq P_k} \theta_{k0}  =1\; \text{a.e. on $I_i, \; \forall i=1, \dots, h$} 
\eq 
and 
\be \label{e:sommaunobd}
    0\leq \bar  \theta_{k} \leq 1,  \; \text{a.e. on $]0, T[, \; \forall k=1, \dots, m$}, \qquad 
    \sum_{k =1}^m \bar \theta_{k}  =1\; \text{a.e. on $]0, T[$}. 
\eq
We now provide the definition of distributional solution of the multi-path model. 
\begin{definition}
\label{d:adsol}  For every $i=1, \dots, h$, $k=1, \dots, m$, fix the data $\rho_{i0} \in L^\infty (I_i), \theta_{k0} \in L^\infty (P_k)$, $\bar \rho, \bar \theta_k \in L^\infty (]0, T[)$ and assume that $\bar \theta_{k0}$ and $\bar \theta_k$ satisfy~\eqref{e:sommaunoid} and~\eqref{e:sommaunobd}, respectively. A distributional solution of the multi-path model is a family of functions $\rho_i \in L^\infty (]0, T[ \times I_i)$, $\theta_k \in L^\infty (]0, T[ \times P_k)$, $i=1, \dots, h$ and $k=1, \dots, m$, such that  
\begin{itemize}
\item[i)] for every $i=1, \dots, h$, $\rho_i$ is an entropy 
admissible solution of~\eqref{e:foglio},~\eqref{e:id1}. Also, $\rho_1$ is an entropy admissible solution of~\eqref{e:foglio},~\eqref{e:id1},~\eqref{e:bd1}, in the sense of~\cite{BLN79};
\item[ii)] Equation~\eqref{e:patches} holds true (that is $r_k$ is obtained by patching together the $\rho_i$-s); 
\item[iii)] $\theta_k$ is a distributional solution of the initial-boundary value problem~\eqref{e:te}, \eqref{e:id2},~\eqref{e:bd2}, in the sense of  Definition~\ref{d:te}. Also, it satisfies 
\be \label{e:zerouno}
      \text{for every  $k=1, \dots, m$}, \quad 0 \leq \rho_i \theta_k \leq \rho_i 
\eq
and
\be 
\label{e:sum1}
          \text{for every $i=1, \dots, h$}, \quad \rho_i \sum_{k: \,  I_i \subseteq P_k} \theta_{k}  =\rho_i 
\eq
a.e. on $]0, T[ \times I_i$. 
\end{itemize}
\end{definition}
Some remarks are here in order. First, in \S\ref{ss:BLN} we recall the definition of entropy admissible solution of~\eqref{e:foglio},~\eqref{e:id1} and of~\eqref{e:foglio},~\eqref{e:id1}~\eqref{e:bd1}.
Second, the heuristic meaning of~\eqref{e:zerouno} and~\eqref{e:sum1} is 
$$
  0 \leq \theta_k \leq 1 \quad \text{a.e. on $\R_+ \times P_k$}, \qquad   \sum_{k: \,  I_i \subseteq P_k}^m \theta_k  = 1 \quad \text{a.e. on $\R_+ \times I_i$},
$$
but owing to~\eqref{e:te} $\theta_k$ is not uniquely defined on the set where $r_k$ vanishes and this is why~\eqref{e:zerouno} and~\eqref{e:sum1} are the correct formulation. Third, distributional solutions of the multi-path model satisfy the flux conservation at junctions, see Lemma~\ref{l:iff}. We can now state our well-posedness result. {\color{black}We explicitly point out that it is an existence and uniqueness result, whereas several other results concerning traffic models on road networks only establish existence, see for instance~\cite{CocliteGaravelloPiccoli,GaravelloMarcellini,GaravelloPiccoli:CMS}}. 
\begin{theorem} \label{t:exunisd}
Fix $T>0$ and assume that $v$ and $g$ satisfy~\eqref{e:v} and~\eqref{e:f}, respectively. For every $i=1, \dots, h$, $k=1, \dots, m$, fix the initial data $\rho_{i0} \in L^\infty (I_i), \theta_{k0} \in L^\infty (P_k)$ and the boundary data $\bar \rho, \bar \theta_k \in L^\infty (]0, T[)$. Assume that $\bar \theta_{k0}$ and $\bar \theta_k$ satisfy~\eqref{e:sommaunoid} and~\eqref{e:sommaunobd}, respectively, and that $0 \leq \bar \rho \leq \rho^\ast$, $0 \leq \rho_{i0} \leq \rho^\ast$. Then there is a  distributional solution of the multi-path model such that $0 \leq \rho_i \leq \rho^\ast$, for every $i=1, \dots, h$. Also, the solution is unique in the following sense: if $\rho_1, \dots ,\rho_h, \theta_1, \dots, \theta_m$ and $\rho^\Diamond_1, \dots ,\rho^\Diamond_h, \theta^\Diamond_1, \dots, \theta^\Diamond_m$ are two solutions such that $0 \leq \rho_i, \rho_i^\Diamond \leq \rho^\ast$ for every $i=1, \dots, h$, then 
$$
    \rho_i = \rho_i^\Diamond, \; \text{a.e. on $]0, T[ \times I_i$, for every $i=1, \dots, h$}
$$
and 
\be \label{e:giugno}
    \rho_i \theta_k = \rho_i \theta_k^\Diamond, 
    \; \text{a.e. on $]0, T[ \times I_i$ for every $i: I_i \subseteq P_k$ and every $k=1, \dots, m$.} 
\eq
\end{theorem}
Some remarks are again in order. First, the uniqueness result given by Theorem~\ref{t:exunisd} is the best one can hope for since, as pointed out before, Equation~\eqref{e:te} does not provide any information on $\theta_k$ on the set where $r_k$ vanishes. Second, as mentioned before, Lemma~\ref{l:iff} states that distributional solutions of the multi-path model satisfy flux conservation at junctions. However, it is well-known that the flux conservation does not suffice to select a unique solution, see the discussion in~\cite{GaravelloHanPiccoli}. The requirement that $0 \leq \rho_i \leq \rho^\ast$, i.e. that the traffic is not congested, can be therefore viewed as an admissibility criterion, which is reasonable in our framework: since the network only contains T-junctions, if the traffic is not congested at the initial time and at the source, one expects that it never gets congested. The next result establishes propagation of $BV$ regularity. 
\begin{theorem} \label{t:propbvreg}
Under the same assumptions as in the statement of Theorem~\ref{t:exunisd}, assume furthermore that, for every $i=1, \dots, h$ and $k=1, \dots, m$, $\rho_{i0} \in BV (I_i), \theta_{k0} \in BV (P_k)$, $\bar \rho, \bar \theta_k \in BV (]0, T[)$. Also, assume that, for some constant $\ee >0$, 
$$\ee \leq \rho_{i0}, \bar \rho \leq \rho^\ast -\ee, 
    \qquad  \ee \leq \theta_{k0}, \bar \theta_k \leq 1,
   \quad \text{for every $i=1, \dots, h$ and $k=1, \dots, m$.}
$$
Then the distributional solution $\rho_1, \dots, \rho_h, \theta_1, \dots, \theta_m$ of the source destination model satisfies $\rho_i \in BV (]0, T[ \times I_i)$ and $\theta_k \in BV (]0, T[ \times P_k)$, for every $i=1, \dots, h$ and $k=1, \dots, m$.
\end{theorem}
Note that, under the assumptions of Theorem~\ref{t:propbvreg}, $r_k$ is bounded away from $0$ and hence the function $\theta_k$ is uniquely determined in view of the uniqueness result given in Theorem~\ref{t:exunisd}. Also, by carefully tracking the proof one could establish, if needed, an explicit bound on the total variation of $\rho_i$ and $\theta_k$, $i=1, \dots, h$, $k=1, \dots, m$, in terms of the total variation of the data and $\ee$. Finally, it is interesting to compare Theorem~\ref{t:propbvreg} with a counterexample in~\cite{BressanYu}.
More precisely,~\cite[Example 4]{BressanYu} show that, on a general network and for initial densities that attain the value $0$, one can have finite time blow up of the total variation even if the data have arbitrarily small total variation. 
To conclude, we establish the $L^1$-stability of the distributional solutions of the source-destination model with respect to perturbations in the data. 
\begin{corol} \label{c:stability}
Fix $T>0$ and assume that $v$ and $g$ satisfy~\eqref{e:v} and~\eqref{e:f}, respectively. For every $i=1, \dots, h$, $k=1, \dots, m$, fix some sequences of initial data $\{ \rho^n_{i0}\}_{n \in \mathbb N} \subseteq L^\infty (I_i), \{ \theta^n_{k0}\}_{n \in \mathbb N} \subseteq L^\infty (P_k)$ and of boundary data $\{ \bar \rho^n \}_{n \in \mathbb N}, \{ \bar \theta^n_k\}_{n \in \mathbb N} \subseteq L^\infty (]0, T[)$ in such a way that, for every $n$, $0 \leq \rho^n_{i0} \leq \rho^\ast$, $0 \leq \bar \rho^n \leq \rho^\ast$ and~\eqref{e:sommaunoid} and~\eqref{e:sommaunobd} are satisfied. Also, assume that 
\be   \label{e:novembre}
     \rho^n_{i0} \to \rho_{i0} \; \text{in $L^1(I_i)$}, \quad
     \theta_{k0}^n \to \theta_{k0} \; \text{in $L^1(P_k)$}, \quad
    \bar \rho^n \to \bar \rho \; \text{in $L^1 (]0, T[)$}, \quad  
     \bar \theta_k^n \to \bar \theta_k \; \text{in $L^1(]0, T[)$}
\eq
for every $i=1, \dots, h$, $k=1, \dots, m$, as $n \to + \infty$. 
Let $\{ \rho_i^n, \theta_k^n \}_{n \in \mathbb N}$ denotes a sequence of distributional solutions of the source-destination model with data $\rho^n_{i0}, \theta^n_{k0}, \bar \rho^n, \bar \theta^n_k$, $i=1, \dots, h$, $k=1, \dots, m$. Then
\be \label{e:agosto}
    \rho^n_i \to \rho_i \; \text{in $L^1 (]0, T[ \times I_i)$}, 
    \quad 
   r_k^n \theta^n_k \to r_k \theta_n \; \text{in $L^1 (]0, T[ \times P_k)$, 
   for every $i=1, \dots, h$ and  $k=1, \dots, m$}.
\eq
In the previous expression,  $\rho_1, \dots, \rho_h, \theta_1, \dots, \theta_m$ is the distributional solution of the source-destination model with data 
$\rho_{i0}, \theta_{k0}, \bar \rho, \bar \theta_k$, $i=1, \dots, h$, $k=1, \dots, m$ and $r^n_k$, $r_k$ are obtained by patching together the $\rho_i^n$-s and the $\rho_i$-s, respectively, see~\eqref{e:patches}. 
\end{corol}
Again, it is interesting to compare Corollary~\ref{c:stability} with a counterexample by Bressan and Yu. More precisely, Example 5 in~\cite{BressanYu} is concerned by a simple network consisting of a single T-junction with an incoming road and two outgoing roads and exhibits instability with respect to the weak$^\ast$ convergence. Note however that a key point in the construction 
of~\cite[Example 5]{BressanYu} is that the flux function in the incoming and outgoing roads is not the same, whereas here we are assuming that  it is the same on every road.

\subsection*{Outline}
The exposition is organized as follows. In \S\ref{s:overview} we overview some previous results that we need in the following. In \S\ref{s:pmain2} we provide the proof of Theorem~\ref{l:tvflux}. In \S\ref{s:other} we establish the proof of Proposition~\ref{p:borraccia2} and of Corollary~\ref{c:stazione} and we discuss the example of total variation blow-up. In \S\ref{s:normal} we complete the distributional formulation of the source-destination model and we establish the proof of Theorem~\ref{t:exunisd}. In \S\ref{ss:propbv} we give the proof of Theorem~\ref{t:propbvreg} and Corollary~\ref{c:stability}. 
\subsection*{Notation} For the reader's convenience, we collect here the main notation used in the present paper. 
We denote by $C(a_1, \dots, a_\ell)$ a constant only depending on the quantities $a_1, \dots, a_\ell$. Its precise value can vary from occurrence to occurrence. 
\subsubsection*{General mathematical symbols} 
\begin{itemize}
\item a.e., for a.e.: almost everywhere, for almost every. Unless otherwise specified, it means with respect to the standard Lebesgue measure; 
\item $BV$: the space of bounded variation functions;
\item $\mathrm{TotVar}\ u$: the total variation of the function $u$;
\item $u(\cdot, x^\pm)$: the left and right trace of the function $u \in BV (]0, T[ \times \R)$ at $y=x$, which are well defined owing to the general theory of $BV$ functions, see~\cite{AmbrosioFuscoPallara}; 
\item $u_\alpha, u_\beta$, also denoted by $u(\cdot, \alpha^+)$, $u(\cdot, \beta^-)$: the \emph{strong} traces given by Theorem~\ref{l:verde};
\item $\mathrm{Tr} [br \theta] (\cdot, \alpha^+)$,
$\mathrm{Tr} [br \theta] (\cdot, \beta^-)$: the \emph{distributional} traces given by Lemma~\ref{l:normaltrace2};
\item $u (\alpha^+), u(\beta^-):$ the right limit of the function $u \in BV (]\alpha, \beta[)$ at $x =\alpha$ and the left limit at $x = \beta$;
\end{itemize}
\subsubsection*{Symbols introduced in the present paper}
\begin{itemize}
\item $I_1, \dots, I_h$: the roads in the source-destination network;
\item $P_1, \dots, P_k$: the paths in the source-destination network;
\item $\rho_i$: the total car density on the road $I_i$;
\item $v$: the velocity function in~\eqref{e:foglio};
\item $\rho^\mathrm{max}$: the maximum possibile car density on $I_i$, see~\eqref{e:v};
\item $g (\rho):= \rho v(\rho)$, see~\eqref{e:f};
\item $\rho^\ast$: the threshold between free and congested traffic, see~\eqref{e:f};
\item $\theta_k$: the traffic-type function in~\eqref{e:te};
\item $r_k$: the function obtained by patching together the $\rho_i$-s,
 see~\eqref{e:patches};
\item $\rho_{i0}, \theta_{k0}$: the initial data in~\eqref{e:id1} and~\eqref{e:id2};
\item $\bar \rho_i$, $\bar \theta_k$: the boundary data in~\eqref{e:bd1} and~\eqref{e:bd2}.
\end{itemize}
\section{Overview of previous results} \label{s:overview}
In this section we collect some previous results that we need in the following.
\subsection{A regularity result for zero-divergence vector fields} \label{ss:daf}
We quote a very special case of Lemma 1.3.3 in \cite{Dafermos:book}.
\begin{lemma} \label{l:dafermos}
       Fix two (finite or infinite) intervals $]0, T[, ]a, b[ \subseteq \R$. Assume that $u, z \in L^\infty (]0, T[ \times ]a, b[ )$ satisfy
       $
          \partial_t u + \partial_x z =0
       $ in the sense of distributions on $]0, T[ \times ]a, b[$. 
       Then  $u$ has a representative such that the map $]0, T[ \to  L^\infty (]a, b[)$, $t \mapsto u(t, \cdot)$ is continuous with respect to 
       the weak$^\ast$ topology. Also,  $z$  has a representative such that the map 
       $]a, b[  \to  L^\infty (]0, T[)$, $x \mapsto z(\cdot, x)$ is continuous with respect to the weak$^\ast$ topology.
\end{lemma}
\begin{remark} \label{r:pointwise}
In the following, we always use the continuous representative of the maps $t \mapsto u(t, \cdot)$ and $x \mapsto z(\cdot, x)$. In this way, the values $u(t, \cdot)$ and $z(\cdot, x)$ are well defined \emph{for every $t$ and $x$}, respectively. 
\end{remark}
\subsection{Continuity of traces}
By combining~\cite[Theorem 3.88]{AmbrosioFuscoPallara} with the observation that translations are continuous
with respect to the strict convergence in $BV$ we get the following result. 
\begin{lemma}\label{L_trace}
Fix $T>0$, assume that $u \in L^\infty \cap BV (]0, T[ \times \R)$ and fix a sequence $\{ x_n \}_{n \in \mathbb N}\subseteq \R$, $x_n \leq \bar x$ such that $x_n \uparrow \bar x \in \R$ as $n \to +\infty$. Then $u(\cdot, x_n^\pm)$ converges in $L^1 (]0, T[)$ to $u(\cdot, \bar x^-)$. If $x_n \ge \bar x$ and $x_n \downarrow \bar x$, then $u(\cdot, x_n^\pm)$ converges in $L^1 (]0, T[)$ to $u(\cdot, \bar x^+)$. 
\end{lemma}

\subsection{Entropy admissible solutions of initial-boundary value problems for scalar conservation laws} \label{ss:BLN} 
We now discuss the definition of entropy admissible solution of the initial-boundary value problem obtained augmenting the conservation law~\eqref{e:cl}
with the initial and boundary conditions 
\be \label{e:data}
    u(0, \cdot) = u_0, \qquad u(\cdot, \alpha) = \bar u, \qquad u (\cdot, \beta) = \underline u.
\eq 
We restrict to the one-dimensional case because it is the one we need in the following, however several results we quote extend to the multi-dimensional setting, see the book by Serre~\cite{Serre} for a general discussion on initial-boundary value problems for conservation laws. 
We first quote (a particular case of) a result due to Kwon and Vasseur~\cite[Theorem 1]{KwonVasseur} and  Panov~\cite[Theorem 1.1]{Panov:traces}. 
\begin{theorem}
\label{l:verde} Fix $T>0$, a bounded interval $]\alpha, \beta[ \subseteq \R$ and a flux function $f \in C^2 (\R)$. Assume that $u \in L^\infty (]0, T[ \times ]\alpha, \beta[)$ 
satisfies
\be \label{e:gennaio}
          \int_0^T \int_{\alpha}^{\beta}  \partial_t \phi |u- c | + \partial_x \phi \ \mathrm{sign}(u- c) [f(u) - f(c)]  dx dt \ge 0, \; \; \text{for every $\phi \in C^\infty_c (]\alpha, \beta[ \times ]0, T[)$, $\phi \ge 0$, 
  $c \in \R$},
\eq    
then there are $u_\alpha, u_\beta \in L^\infty (]0, T[)$ such that 
\be \label{e:febbraio}  
      \esslim_{y \to \alpha^+} \int_0^T |h(u)(\cdot, y) - h(u_\alpha) | dt =0 \quad \text{and} \quad 
       \esslim_{y \to \beta^-} \int_0^T |h(u)(\cdot, y) - h(u_\beta ) | dt =0
\eq
for $h(u): =f(u)$ and $h(u) : = \mathrm{sign}(u- c) [f(u) - f(c)]$, $c \in \R$. 
\end{theorem}
In the following  we will refer to $f(u_\alpha)$ and $f(u_\beta)$ as the \emph{strong traces} of $f(u)$, see also Remark~\ref{r:giallo}. Also, we will sometimes denote them by $f(u)(\cdot, \alpha^+)$ and  $f(u)(\cdot, \beta^-)$, respectively. Note that in the previous result the regularity of $u$ is only $u \in L^\infty(]0, T[ \times ]\alpha, \beta[)$ and that in general one cannot replace $h$ with the identity in~\eqref{e:febbraio}. However, this is instead possible if $u \in BV (]\alpha, \beta[ \times ]0, T[)$ owing to the general theory of $BV$ functions, see~\cite{AmbrosioFuscoPallara}, or when $f''>0$ or $f''<0$, see~\cite{Vasseur}.  
\begin{definition}\label{d:bln}
Fix $f \in C^2 (\R)$, $\alpha, \beta \in \R $, $T>0$, $\bar u, \underline u \in L^\infty (]0, T[)$ and $u_0 \in L^\infty (]\alpha, \beta[)$. We say that $u \in L^\infty (]0, T[ \times ]\alpha, \beta[) $ is an entropy admissible solution of the initial-boundary value problem~\eqref{e:cl}, \eqref{e:data} if 
 \begin{equation}
\label{e:bln}
\begin{split}
        \int_0^T \int_{\alpha}^{\beta} & \partial_t \varphi |u- c | + \partial_x \varphi \ \mathrm{sign}(u- c) [f(u) - f(c)]  dx dt + 
        \int_{\alpha}^{\beta} \varphi (0, \cdot) |u_0 - c| dx \\ & 
  +   \int_0^T \varphi (\cdot, \alpha) \mathrm{sign}(\bar u -c) [f (u_\alpha) - f(c)] dt -
   \int_0^T \varphi (\cdot, \beta) \mathrm{sign}(\underline u -c) [f (u_\beta)  - f(c)] dt \ge 0 
\end{split}
\end{equation}
for every $c \in \R$ and $\varphi \in C^\infty_c (]- \infty, T[ \times \R)$ such that $\varphi \ge 0$.
\end{definition}
The above definition should be interpreted in the following sense: if~\eqref{e:bln} holds true, then in particular~\eqref{e:gennaio} is satisfied, and hence by Theorem~\ref{l:verde} the values $f(u_\alpha)$ and $f(u_\beta)$ are well defined and satisfy~\eqref{e:febbraio}. Note that an alternative approach to provide a definition of entropy admissible solution of~\eqref{e:cl}, \eqref{e:data} is discussed by Otto~\cite{Otto}. See also the discussion in~\cite{Rossi}. 

The analysis in~\cite{BLN79} combined with Theorem~\ref{l:verde} yields existence and uniqueness results for the entropy admissible solution of ~\eqref{e:cl}, \eqref{e:data}.
Also, the entropy admissible solution satisfies the maximum principle: if  
$\kappa \leq \bar u, \underline u, u_0 \leq K$ a.e., for some constants $\kappa, K \in \R$ then 
\be
   \label{e:maxprincl}
   \kappa \leq  u \leq K \quad \text{a.e. on $]0, T[ \times ]\alpha, \beta[$}. 
\eq
Also, if $\bar u, \underline u \in BV (]0, T[)$ and $u_0 \in BV (]\alpha, \beta[)$, then 
$$
      \mathrm{Tot Var} \ u(t, \cdot) \leq \mathrm{Tot Var} \ u_0 + \mathrm{Tot Var} \  \bar u 
      +  \mathrm{Tot Var} \  \underline u +
      |u_0 (\alpha^+) - \bar u(0^+) |  +
       |u_0 (\beta^-) - \underline u(0^+) |
$$
for every $t>0$. In the previous expression, $u_0 (\alpha^+)$ and 
       $u_0 (\beta^-)$ denote the right limit of the function $u_0$ at $x=\alpha$ and the left limit at $x = \beta$, respectively. They are well defined since $u_0 \in BV(]\alpha, \beta[)$. We now focus on the case  $f' \ge 0$, then to obtain a well-posed problem it suffices to assign the data at $t=0$ and $x = \alpha$, that is 
\be \label{e:data2}
     u(0, \cdot) = u_0, \qquad u(\cdot, \alpha)= \bar u. 
\eq
More precisely, we have the following. 
\begin{proposition} \label{p:datainc}
Fix $T>0$, a bounded interval $]\alpha, \beta[\subseteq \R$ and a flux function $f \in C^2 (\R)$. 
Assume furthermore  that $f'\ge 0$ on $[\min \{ \essinf \bar u, \essinf u_0 \}, \max  \{ \essup \bar u, \essup u_0 \}]$. Then there is a unique entropy admissible solution of the initial-boundary value problem~\eqref{e:cl}, \eqref{e:data2} such that $f'(u) \ge 0$ a.e. on $]0, T[ \times ]\alpha, \beta[$. In other words, there is a unique function $u \in L^\infty (]0, T[ \times ]\alpha, \beta)$ such that $f'(u)\ge0$ and 
\begin{equation}
\label{e:bln2}
\begin{split}
        \int_0^T \int_{\alpha}^{\beta}  \partial_t \varphi |u- c |  + &
       \partial_x \varphi \ \mathrm{sign}(u- c) [f(u) - f(c)]  dx dt + 
        \int_{\alpha}^{\beta} \varphi (0, \cdot) |u_0 - c| dx \\ & 
  +   \int_0^T \varphi (\cdot, \alpha) \mathrm{sign}(\bar u -c) [f (u_\alpha) - f(c)] dt  \ge 0 
\end{split}
\end{equation}
for every $c \in \R$ and $\varphi \in C^\infty_c (]- \infty, T[ \times ]- \infty, \beta[)$ such that $\varphi \ge 0$.
\end{proposition}
Note that, under the same assumptions as in Proposition~\ref{p:datainc}, if $\bar u \in BV(]0, T[)$ and $u_0 \in BV (]\alpha, \beta[)$, then 
\be 
\label{e:giallo}
      \mathrm{Tot Var} \ u(t, \cdot) \leq \mathrm{Tot Var} \ u_0 + \mathrm{Tot Var} \  \bar u 
      + 
      |u_0 (\alpha^+) - \bar u(0^+) | , \quad \text{for every $t \ge 0$.}
\eq
The following result is well known and we provide the proof for the sake of completeness. 
\begin{proposition}
\label{p:datumclassic}
Fix $T>0$, a bounded interval $]\alpha, \beta[\subseteq \R$ and a flux function $f \in C^2 (\R)$. 
Assume furthermore  that $f' > 0$ on $[\min \{ \essinf \bar u, \essinf u_0 \}, \max  \{ \essup \bar u, \essup u_0 \}]$
and that $u \in L^\infty (]0, T[ \times ]\alpha, \beta[)$ satisfies~\eqref{e:bln2}. Then $u_\alpha = \bar u$ a.e. on $]0, T[$.
\end{proposition}
\begin{proof}
We fix $\lambda \in C^\infty_c (]0, T[)$ and a family of functions $\omega_\ee \in C^\infty_c (\R)$ such that 
\be \label{e:cremisi2}
    \omega_\ee(x) =1 \; \text{if $\alpha -1<x < \alpha + \ee$}, \quad
    \omega_\ee (x) =0 \; \text{if $x > \alpha + 2 \ee$}, 
    \quad \omega'_\ee (x) \leq 0 \; \text{for every $x \in ]\alpha -1, + \infty[$}. 
\eq
We plug the test function $\varphi_\ee (t, x): = \lambda (t) \omega_\ee (x)$ into~\eqref{e:bln2} and let $\ee \to 0^+$. By relying on Lemma~\ref{l:verde} we obtain 
$$
     \int_0^T \lambda  [ \mathrm{sign}(\bar u -c) -\mathrm{sign}(u_\alpha -c)]
     [f (u_\alpha) - f(c)] dt \ge 0
$$
and by the arbirariness of $\lambda$ this implies that $ [ \mathrm{sign}(\bar u -c) -\mathrm{sign}(u_\alpha -c)]
     [f (u_\alpha) - f(c)] \ge 0$ a.e. on $]0, T[$. By relying on a case-by-case analysis we can then conclude that $u_\alpha = \bar u$ a.e. on $]0, T[$. 
\end{proof}
In the following we also need the next result and again we provide a sketch of the proof for the sake of completeness.
\begin{lemma}\label{l:malva}
Under the same assumptions as in the statement of Proposition~\ref{p:datainc}, assume that $\{ u_{0n} \} \subseteq L^\infty (]\alpha, \beta[)$ and $\{ \bar u_{n} \}\subseteq L^\infty (]0, T[)$ are two sequences of initial and boundary data such that 
$$
   f' (  u_{0n}  ) \ge 0 \quad \text{a.e. on $ ]\alpha, \beta[$}, 
   \quad 
   f' (\bar u_{n} ) \ge 0 \quad \text{ a.e. on $ ]0, T[$, for every $n \in \mathbb N$}  
$$
and 
\be \label{e:malva2}
     u_{0n} \to u_0 \quad \text{in $L^1 (]\alpha, \beta[)$}, 
     \quad 
    \bar u_{n} \to \bar u \quad \text{in $L^1 (]0, T[)$}.
\eq
Let $\{ f(u_{\beta n}) \}$ be the sequences of the traces of the fluxes as in Theorem~\ref{l:verde}, then $f (u_{\beta n})$ converges to $f(u_\beta)$ in $L^1(]0, T[)$. 
\end{lemma}
\begin{proof}[Sketch of the proof of Lemma~\ref{l:malva}]
Let $u$ be the entropy admissible solution of the initial-boundary value problem~\eqref{e:cl},\eqref{e:data2}. By relying on~\eqref{e:bln2} and on a suitable choice of the test functions (in the same spirit as the one in the proof of Proposition~\ref{p:datumclassic}) and by recalling Theorem~\ref{l:verde} we arrive at 
\begin{equation*}
\begin{split}
    \int_0^T \underbrace{ \mathrm{sign}(u_\beta-c)[f(u_\beta) - f(c) ]}_{= |f(u_\beta) - f(c)| \; \text{since $f'\ge 0$}} dt  & \leq 
    \int_\alpha^\beta |u_0 -c | dx
    +  \int_0^T  [ \mathrm{sign}(\bar u -c) -\mathrm{sign}(u_\alpha -c)]
     [f (u_\alpha) - f(c)]dt\\ &
     \leq 
    \int_\alpha^\beta |u_0 -c | dx
    + 2 \int_0^T |
     f (u_\alpha) - f(c)| dt. 
\end{split}
\end{equation*}
By using the celebrated doubling of variables method by Kru{\v{z}}kov~\cite{Kruzkov} and recalling that, by the same argument as in Proposition~\ref{p:datumclassic}, $f(u_\alpha) = f(\bar u)$, $f(u_{\alpha n}) = f(\bar u_n)$ we arrive at 
$$
    \int_0^T | f(u_\beta) - f(u_{\beta n})| dt  \leq 
    \int_\alpha^\beta |u_0 -u_{0 n} |  dx
    + 2 \int_0^T|f (u_\alpha) - f(\bar u_{n})| dt,
$$
and owing to~\eqref{e:malva2} this yields the convergence of $f(u_{\beta n})$ to $f(u_{\beta})$. 
\end{proof}
By relying on the proof of Lemma~\ref{l:malva} we also get the following result. 
\begin{lemma}
\label{l:lanacaprina}
Fix $T>0$, a bounded interval $]\alpha, \beta[\subseteq \R$ and a flux function $f \in C^2 (\R)$.  Assume that $u, v \in L^\infty (]0, T[ \times ]\alpha, \beta[)$ are two entropy admissible solutions in the sense of Definition~\ref{d:bln} and satisfy the inequality $f'  \ge 0$ a.e. on $[\min \{\essinf v, \essinf u \}, \max \{\essup v, \essup u \} ]$. Assume that $u_0 = v_0$ and $f(\bar u) = f(\bar v)$, then $u =v$ a.e. on $]0, T[ \times ]\alpha, \beta[$.
\end{lemma}

To conclude, we assign the initial condition
\be \label{e:lilla}
     u(0, \cdot) = u_0
\eq
and, since for technical reasons we need it in the following, we give the definition of entropy admissible solution of~\eqref{e:cl},\eqref{e:lilla}. 
\begin{definition}
\label{d:giallo}
       Fix $T>0$, an interval $]\alpha, \beta[ \subseteq \R$ and a flux function $f \in C^2 (\R)$. We say that $u\in L^\infty(]0, T[\times ]\alpha, \beta[)$ is an entropy admissible solution of~\eqref{e:cl},\eqref{e:lilla} if 
\begin{equation}
\label{e:aprile}
\begin{split}
        \int_0^T \int_{\alpha}^{\beta} & \partial_t \psi |u- c | + \partial_x \psi \ \mathrm{sign}(u- c) [f(u) - f(c)]  dx dt + 
        \int_{\alpha}^{\beta} \psi (0, \cdot) |u_0 - c| dx \ge 0 
\end{split}
\end{equation}
for every $\psi \in C^\infty_c (]-\infty, T[ \times ]\alpha, \beta[$) such that $\psi \ge 0$ and $c \in \R$. 
\end{definition}
Needless to say, unless $]\alpha, \beta[ = \R$, in general we do not expect that the entropy admissible solution of~\eqref{e:cl},\eqref{e:lilla} is unique because we are not prescribing any boundary condition. 
\subsection{Distributional traces for solutions of continuity equations}
\label{ss:acm}
In the following we quote a result that will enable us to give a meaning to the boundary condition in~\eqref{e:bd2}. We refer to~\cite{AmbrosioCrippaManiglia,Anzellotti,ChenZiemerTorres} for a general discussion about normal traces for measure divergence vector fields. We now state a straighforward corollary of~\cite[Lemma 3.3]{CrippaDonadelloSpinolo}.
\begin{lemma}
\label{l:normaltrace2}
Fix $\alpha, \beta \in \R, T>0$. Assume that $r, b, \theta \in L^\infty (]0, T[ \times ]\alpha, \beta[)$ satisfy 
\be \label{e:dsoltheta}
        \int_0^T \! \! \int_\alpha^\beta r \theta  (\partial_t \phi +  b  \partial_x \phi ) 
      dx dt     = 0
    \quad \text{for every $\phi \in C^\infty_c (]0, T[ \times ]\alpha, \beta[)$.}
\eq
Then there are unique functions $\mathrm{Tr} [br \theta] (\cdot, \alpha^+), \mathrm{Tr} [b r  \theta] (\cdot, \beta^-)~\in~L^\infty (\R_+)$ and $[r \theta ]_0~\in~L^\infty (]\alpha, \beta[)$ such that 
\be
\label{e:suppout42}
    \begin{split}
    \int_0^T \! \! \int_{\alpha}^\beta r \theta (\partial_t \varphi +  b \partial_x \varphi) dx dt & = 
    \int_0^T \! \!  \varphi \mathrm{Tr} [b r \theta] (\cdot, \alpha^+)  dt + \int_0^T \! \!  \varphi \mathrm{Tr} [b r\theta] (\cdot, \beta^-)  dt \\
    & \quad -
    \int_\alpha^\beta [r \theta]_0 \varphi(0, \cdot) dx , 
    \quad \text{for every $\varphi \in C^\infty_c (]- \infty, T[ \times \R)$.}
\end{split}
\eq
\end{lemma}
\begin{proof}
We apply~\cite[Lemma 3.3]{CrippaDonadelloSpinolo} with $w= r\theta$ and $b$ in~\cite[Lemma 3.3]{CrippaDonadelloSpinolo} given by $b r \theta$ in here and we point out that to establish~\cite[formula (3.6)]{CrippaDonadelloSpinolo} we do not use the assumption that $\mathrm{div}\, b$ is a finite Radon measure. 
\end{proof}
\begin{remark}\label{r:giallo}
From now on we refer to $\mathrm{Tr} [br \theta] (\cdot, \alpha^+)$ and 
$\mathrm{Tr} [br \theta] (\cdot, \beta^-)$ as \emph{distributional traces}, whereas we refer to the functions $f(u_\alpha)$, $f(u_\beta)$ given by Theorem~\ref{l:verde} as \emph{strong traces}. The reason is the following: the functions $f(u_\alpha)$, $f(u_\beta)$ are \emph{strong} traces in the sense that they are attained as strong limits in the $L^1$ topology, whereas $\mathrm{Tr} [br \theta] (\cdot, \alpha^+)$ and 
$\mathrm{Tr} [br \theta] (\cdot, \beta^-)$ are \emph{distributional} traces and in general they are only attained as limits in the weak$^\ast$ topology, see the discussion in~\cite{AmbrosioCrippaManiglia}.  Also, we recall that in the following we will sometimes denote $f(u_\alpha)$ and $f(u_\beta)$ by $f(u)(\cdot, \alpha^+)$ and $f(u)(\cdot, \beta^-)$, respectively. 
\end{remark}
The next result asserts that the two notions of traces coincide when they are both defined. 
\begin{lemma} \label{l:maggio}
Fix $\alpha, \beta \in \R, T>0$, $v \in C^1(\R)$ and assume that $\rho_i \in L^\infty(]0, T[\times ]\alpha, \beta[)$ satisfies~\eqref{e:gennaio} with $u= \rho_i$, $f(\rho_i)= v(\rho_i) \rho_i$, then 
\be \label{e:maggio}
     v(\rho_i) \rho_i (\cdot, \alpha^+)= - \mathrm{Tr} [  v(\rho_i) \rho_i ] (\cdot, \alpha^+),
     \qquad 
     v(\rho_i) \rho_i (\cdot, \beta^-)= \mathrm{Tr} [  v(\rho_i) \rho_i ] (\cdot, \beta^-).  
\eq
\end{lemma}
Note that the distributional traces $\mathrm{Tr} [  v(\rho_i) \rho_i ] (\cdot, \alpha^+)$ and $ \mathrm{Tr} [  v(\rho_i) \rho_i ] (\cdot, \beta^-) $ are well defined since by choosing $c < - \| \rho_i \|_{L^\infty}$ and $c > \| \rho_i \|_{L^\infty}$ we deduce from~\eqref{e:gennaio} that $r=\rho_i$ satisfies~\eqref{e:dsoltheta}  with $\theta=1$ and $b = v(\rho_i)$ and hence we can apply Lemma~\ref{l:normaltrace2}. 
\begin{proof}[Proof of Lemma~\ref{l:maggio}]
We fix $\varphi \in C^\infty (]0, T[ \times \R)$ and consider the family of test functions 
$\phi_\ee(t, x) : = \varphi(t, x)[1-\omega_\ee (x)][1 - \omega_\ee (\alpha + \beta - x)]$, where $\omega_\ee$ is the same as in~\eqref{e:cremisi2}. We plug $\phi_\ee$ into~\eqref{e:dsoltheta} where $r=\rho_i$, $\theta=1$ and $b = v(\rho_i)$ and then let $\ee \to 0^+$. By using~\eqref{e:febbraio} we arrive at 
\begin{equation*}
    \begin{split}
    \int_0^T \! \! \int_{\alpha}^\beta \rho_i (\partial_t \varphi +  v(\rho_i) \partial_x \varphi) dx dt & = 
   - \int_0^T \! \!  \varphi \ v(\rho_i) \rho_i  (\cdot, \alpha^+)  dt + \int_0^T \! \!  \varphi \ v(\rho_i) \rho_i  (\cdot, \beta^-)  dt 
\end{split}
\end{equation*}
and by the arbitrariness of $\varphi$ this yields~\eqref{e:maggio}.
\end{proof}

The following result is well known and we provide the proof for the sake of completeness. 
\begin{lemma} \label{l:orecchie}
Under the same assumptions as in Lemma~\ref{l:normaltrace2}, let $d \in ]\alpha, \beta[$. Then 
\be \label{e:mirtillo}
       \mathrm{Tr} [b r \theta] (\cdot, d^+)= -  \mathrm{Tr} [b r \theta] (\cdot, d^-).
\eq
\end{lemma}
In the above formula the functions $\mathrm{Tr} [b r \theta] (\cdot, d^+)$ and $\mathrm{Tr} [b r \theta] (\cdot, d^-)$ are obtained by applying Lemma~\ref{l:normaltrace2} to the intervals $]d, \beta[$ and $]\alpha, d[$, respectively. 
\begin{proof}
We fix $\varphi \in C^\infty_c (]- \infty, T[ \times \R)$, then 
\begin{equation*}
     \begin{split}
    \int_0^T \! \! \int_{\alpha}^d r \theta (\partial_t \varphi +  b \partial_x \varphi) dx dt & = 
    \int_0^T \! \!  \varphi \mathrm{Tr} [b r \theta] (\cdot, \alpha^+)  dt + \int_0^T \! \!  \varphi \mathrm{Tr} [b r\theta] (\cdot, d^-)  dt -
    \int_\alpha^d [r \theta]_0 \varphi(0, \cdot) dx 
\end{split}
\end{equation*}
and 
\begin{equation*}
     \begin{split}
    \int_0^T \! \! \int_{d}^\beta r \theta (\partial_t \varphi +  b \partial_x \varphi) dx dt & = 
    \int_0^T \! \!  \varphi \mathrm{Tr} [b r \theta] (\cdot, d^+)  dt + \int_0^T \! \!  \varphi \mathrm{Tr} [b r\theta] (\cdot, \beta^-)  dt  -
    \int_d^\beta [r \theta]_0 \varphi(0, \cdot) dx. 
\end{split}
\end{equation*}
By adding the above expressions, recalling~\eqref{e:suppout42} and using the arbitrariness of the test function $\varphi$ we arrive at~\eqref{e:mirtillo}. 
\end{proof}
\subsection{Well-posedness of the initial-boundary value problem for nearly incompressible vector fields in one space dimension}
We now quote some results from~\cite{DovettaMarconiSpinolo} we need in the following. 
We first recall that $b \in L^\infty(]0, T[ \times ]\alpha, \beta[)$ is a \emph{nearly incompressible vector field} if there is a function $\rho \in  L^\infty(]0, T[ \times ]\alpha, \beta[)$, $\rho \ge 0$, such that 
$
    \partial_t \rho + \partial_x [b \rho ]=0. 
$
We refer to~\cite{DL07} for an extended discussion on (possibly multi-dimensional) nearly incompressible vector fields. We assume that $b \ge 0$ and consider the initial-boundary value problem
\be \label{e:luglio}
     \left\{
     \begin{array}{ll}
                  \partial_t [\rho \theta] + \partial_x [b \rho \theta]=0 \\
                  \theta(0, \cdot) = \theta_0 \qquad \theta(\cdot, \alpha) = \bar \theta, \\
      \end{array}
     \right.
\eq
where the boundary condition is attained in the sense of~\cite[Definition 2.7]{DovettaMarconiSpinolo}, that is by requiring that 
\be \label{e:luglio2}
\mathrm{Tr}[b \rho \theta] (\cdot, \alpha^+) = \bar \theta \ \mathrm{Tr}[b \rho] (\cdot, \alpha^+).
\eq 
The above distributional traces are well defined owing to Lemma~\ref{l:normaltrace2}, see also the discussion in~\cite[\S2.3]{DovettaMarconiSpinolo}. By combining Theorem 1.2, Corollary 1.3 and Remark 6.2 in~\cite{DovettaMarconiSpinolo} we arrive at the following result. 
\begin{theorem}
\label{t:luglio}
Fix $T>0$ and a bounded from below interval $]\alpha, \beta[\subseteq \R$. 
Let $b \in L^\infty(]0, T[ \times ]\alpha, \beta[)$ be a nearly incompressible vector field with density $\rho$ and assume furthermore that $b\ge 0$. For every $\theta_0 \in L^\infty (]\alpha, \beta[)$ and $\bar \theta \in L^\infty (]0, T[)$ there is a distributional solution $\theta \in L^\infty (]0, T[ \times ]\alpha, \beta[)$ of the initial-boundary value problem~\eqref{e:luglio}, in the sense of~\cite[Definition 2.7]{DovettaMarconiSpinolo}. Also, the solution is unique in the following sense: if $\theta_a, \theta_b  \in L^\infty (]0, T[ \times ]\alpha, \beta[)$
are two different solutions, then $\rho \theta_a = \rho \theta_b$ a.e. on $]\alpha, \beta[ \times ]0, T[$. Finally, we have a comparison principle: if $\theta_{01} \ge \theta_{02}$ and $\bar \theta_1 \ge \bar \theta_2$, then the corresponding solutions satisfy $\rho \theta_1 \ge \rho \theta_2$ a.e. on $]\alpha, \beta[ \times ]0, T[$. 
\end{theorem}
Note that the uniqueness result in Theorem~\ref{t:luglio} is the best one can hope for since the equation at the first line of~\eqref{e:luglio} does not provide any information on $\theta$ on the set where $\rho$ vanishes.

\section{Proof of Theorem~\ref{l:tvflux}} \label{s:pmain2}
We now provide the proof of Theorem~\ref{l:tvflux}. Throughout all the proof, we denote by $u$ the entropy admissible solution of the Cauchy problem~\eqref{e:cl},~\eqref{e:cpdatum} and use the notation $w= f \circ u$. The exposition is organized as follows: in \S\ref{ss:prelim} we establish some preliminary results and in \S\ref{ss:proof} we complete the proof. We always focus on the case $f''\ge 0$. The case $f''\ge0$ follows by recalling that, if $u$ satisfies~\eqref{e:cl}, then $z(t, x): = u(t, -x)$ satisfies
$$
     \partial_t z + \partial_x [-f(z)] =0. 
$$
Also, without loss of regularity we can assume that $f''<0$. The general case $f'' \leq 0$ can be recovered by considering the sequence $f_\ee (u) : = f (u) - \ee u^2$ and then passing to the $\ee \to 0^+$ limit by arguing as in {\sc Step 3D} of \S\ref{ss:proof}. 
\subsection{Preliminary results} \label{ss:prelim}{\color{black}
\begin{lemma} \label{l:upward} 
Assume that $\gamma: ]t_1, t_2[ \to \R$ is a $C^1$ curve across which $u$ is discontinuous.  Fix $\tau \in ]t_1, t_2[$ and assume that there is a neighborhood $\mathcal U$ of $\big(\tau, \gamma (\tau) \big)$ such that there are continuous extensions of the entropy admissible solution $u$ to 
$\mathcal U \cap \{ (t, x): x \leq \gamma (t)\}$ and $\mathcal U \cap \{ (t, x): x \ge \gamma (t)\}$. If $\gamma'(\tau) \neq 0$, then 
\be
\label{e:upward}
        \lim_{t \to \tau^+}w(t, \gamma(\tau)) \leq \lim_{t \to \tau^-}w(t, \gamma(\tau))
\eq  
and, in particular, the above limits are well defined. If $\gamma'(\tau) = 0$, then $w$ is continuous at $(\tau, \gamma(\tau))$. 
\end{lemma}
We point out that, by assumption, $\tau$ is neither the starting nor the final point of the curve $\gamma$. 
\begin{proof}[Proof of Lemma~\ref{l:upward}]
We set 
$$
    u^- : =      \lim_{x \to \gamma ( \tau)^-} u(t, x), \qquad 
    u^+ : =   \lim_{x \to \gamma ( \tau)^+}u(t, x),
$$
and note that the above limits exist since $u(\tau, \cdot) \in BV (\R)$. 
The Rankine-Hugoniot condition gives 
\be \label{e:rh} 
     f(u^+) - f (u^-) = \gamma' (\tau) [u^+ - u^-]
\eq
Also, the Lax admissibility condition yields $f'( u^-) \ge f'(u^+)$ and by the condition $f'' <0$ this implies $u^- \leq u^+$.

By recalling the condition defining $\mathcal U$, we conclude that 
both the limits $ \lim_{t \to \tau^+}w(t, \gamma(\tau))$ and  
       $ \lim_{t \to \tau^-}w(t, \gamma(\tau)) $ exist. If $\gamma'(\tau)>0$, then 
$$
        \lim_{t \to \tau^+}w(t, \gamma(\tau)) = f(u^-),  \qquad 
        \lim_{t \to \tau^-}w(t, \gamma(\tau)) = f(u^+)
$$
and by using the condition~\eqref{e:rh} we arrive at~\eqref{e:upward}.  If $\gamma'(\tau) < 0$ the analysis is similar. If $\gamma'(\tau)=0$ then $w$ is continuous at $(\tau, \gamma (\tau))$. 
\end{proof}}
\begin{lemma}
\label{l:key}
Assume $u_0 \in C^\infty_c (\R)$. Fix $x \in \R$ and a time interval $]0, T[$ and assume that $u$ is of class $C^1$ in a neighborhood of every point $(t, x)$ except for a finite number of points $(\tau_1, x), \dots, (\tau_\ell, x)$. We also assume that, for every $r=1, \dots, \ell$, {\color{black}the entropy admissible solution 
$u$ satisfies the assumptions of Lemma~\ref{l:upward} for a suitable $C^1$ curve $\gamma_r: ]t_{1r}, t_{2r}[ \to \R$ such that $\tau_r \in ]t_{1r}, t_{2r}[$ and $\gamma_r (\tau_r) =x$. }Then
\be
\label{e:tvm}
     \sup_{0 \leq t_1 < t_2 < \dots< t_p \leq T} \sum_{\alpha=1}^{p-1} 
     \big[ w(t_{\alpha+1}, x ) - w(t_{\alpha}, x )\big]^+ \leq C( \mathrm{Tot Var \, } w_0). 
\eq
In the previous expression, $[ \cdot ]^+$ denotes the positive part and $w_0: = f \circ u_0$\,.
\end{lemma} 
Note that, owing to the fact that $w(\cdot, x)$ is smooth outside $\tau_1, \dots, \tau_\ell$, in computing the supremum in~\eqref{e:tvm} we can assume without loss of generality that the sampling points do not coincide with any discontinuity point, that is $t_\alpha \neq \tau_r$, for every $\alpha =1, \dots, p$ and every $r=1, \dots, \ell$. This yields that the value $w(t_{\alpha}, x )$ is well defined. 
\begin{proof}[Proof of Lemma~\ref{l:key}]
We fix a sampling $t_1, \dots, t_p$ and by relying on the above observation we assume without loss of generality that $t_\alpha \neq \tau_r$, for every $\alpha =1, \dots, p$ and every $r=1, \dots, \ell$. To establish~\eqref{e:tvm} we use  the theory of so-called \emph{generalized characteristics} and we refer to~\cite{Dafermos:book} for a comprehensive introduction. The rest of the proof is organized according to the following steps.\\
{\sc Step 1:} we establish some properties of generalized characteristics that we need in the following. We term $\xi^-_{t}$ and $\xi^+_t$ the minimal backward and the maximal backward characteristic emanating from the point $(t, x)$, see~\cite[Theorem 10.2.2]{Dafermos:book}. We apply~\cite[Theorem 10.3.1]{Dafermos:book} and we conclude that, for every $t \in \R_+$,  $\xi_t^-$ and $\xi_t^+$ are a left and a right contact, respectively, in the sense of~\cite[Definition 10.2.5]{Dafermos:book}. {\color{black}By applying~\cite[Formula (11.1.10)]{Dafermos:book} in virtue of the fact that $f'' <0$,} we conclude that $\xi_t^-$ and $\xi^+_t$ are \emph{shock-free}, in the sense of~\cite[Definition 10.2.4]{Dafermos:book}. We apply~\cite[Theorem 11.1.1]{Dafermos:book} and conclude that, for every $t \in \R_+$, $\xi^-_t$ and $\xi^+_t$ are segments with constant slope.

Next, we fix $\alpha =1, \dots, p$, we recall that $t_\alpha \neq \tau_r$, for every $r=1, \dots, \ell$ and by using~\cite[Theorem 11.1.3]{Dafermos:book}  we conclude that $\xi^-_{t_\alpha} \equiv \xi^+_{t_\alpha}$. Owing to~\cite[Theorem 10.2.2]{Dafermos:book} we conclude that there is a unique backward characteristic emanating from the point $(t_\alpha, x)$ and we denote it by $\xi_{t_\alpha}: \R_+ \to \R$. We recall that, by the previous analysis, $\xi_{t_\alpha}$ is shock-free, in the sense of~\cite[Definition 10.2.4]{Dafermos:book}.  Finally, we apply \cite[Theorem 11.1.1]{Dafermos:book} and we conclude that $u$ (and henceforth $w$) is constant along $ \xi_{t_\alpha}$. \\
{\sc Step 2:} we set 
\be \label{e:epm}
    E^- := \big\{ t \in ]0, T[: \; \xi_t^- (0) \leq x \big\},
    \qquad  
    E^+ : = \big\{ t \in  ]0, T[: \; \xi_t^+ (0) \ge x \big\},
\eq 
we point out that $]0,T[ = E^- \cup E^+$ and we establish the following property: the maps $t \mapsto \xi_t^- (0)$ and  $t \mapsto \xi_t^+ (0)$ are monotone non increasing on $E^-$ and monotone non decreasing on $E^+$, respectively.

We show that  $t \mapsto \xi_t^+ (0)$ is monotone non decreasing on $E^+$ , the proof of the other claim is analogous.  Assume by contradiction that there are $t_1, t_2 \in E^+$, $t_1 < t_2$, such that $\xi^+_{t_1} (0) > \xi^+_{t_2}(0)$.  Owing to {\sc Step 1}, $\xi^+_{t_1}$ and $\xi^+_{t_2}$ both have constant slope. Since $t_1< t_2$ and $\xi^+_{t_1} (0) > \xi^+_{t_2}(0)$, then $\xi_{t_1} $ and  $\xi_{t_2}$ must cross at some $s< t_1$. Since, by {\sc Step 1}, $\xi^+_{t_1}$ and $\xi^+_{t_2}$ are both shock free, this contradicts~\cite[Corollary 11.1.2]{Dafermos:book} and hence concludes the proof of the claim. \\
{\sc Step 3:} since $]0, T[ = E^- \cup E^+$, then 
\be \label{e:split}
     \sup_{0 \leq t_1 < t_2 < \dots< t_p \leq T}\sum_{\alpha=1}^{p-1} 
     \big[ w(t_{\alpha+1}, x ) - w(t_{\alpha}, x )\big]^+ \leq S_1 + S_2 + S_3,
\eq
where $S_1$ is the supremum of the sum over the $\alpha$-s such that $t_{\alpha+1}$ and  $t_\alpha$ both belong to $E^-$,  $S_2$ is the supremum of the sum over the $\alpha$-s such that $t_{\alpha+1}$ and  $t_\alpha$ both belong to $E^+$ and $S_3$ is the supremum of the sum over the $\alpha$-s such that $t_{\alpha+1} \in E^-$ and  $t_\alpha \in E^+$, or viceversa. 

We first control the term $S_1$ in~\eqref{e:split}.  
We fix $\alpha$ such that $t_{\alpha+1}, t_\alpha \in E^-$. We recall that owing to {\sc Step 1} the backward characteristics emanating from $(t_\alpha, x)$ and $(t_{\alpha+1}, x)$ are both unique and that the function $w$ is constant along $\xi_{t_{\alpha +1} }$ and $\xi_{t_{\alpha} }$. This implies  
\begin{equation*}
   [w(t_{\alpha+1}, x ) - w(t_{\alpha}, x )]^+  = [ w (\xi_{t_{\alpha+1}}(0)) - w(\xi_{t_\alpha} (0)) ]^+  \leq | w_0 (\xi_{t_{\alpha+1}}(0)) - w_0(\xi_{t_{\alpha}}(0)) |  . 
\end{equation*}
Next, we sum all the above contributions for $t_\alpha, t_{\alpha+1} \in E^-$ and we recall that the map $t \mapsto \xi^-_{t}(0)$ is monotone on $E^-$. This implies that when summing we are never computing twice the same interval and hence we eventually arrive at 
\be \label{e:esse1}
   S_1 \leq \mathrm{TotVar} \, w_0.
\eq
By analogous considerations 
$S_2 \leq\ \mathrm{TotVar} \, w_0$.  \\
{\sc Step 4:} we control the term $S_3$. We fix $\alpha$ and, just to fix the ideas, we assume that $t_{\alpha+1} \in E^+$, $t_\alpha \in E^-$. We set
$$
   s_\alpha : = \sup \{s \leq t_{\alpha+1}: \; s \in E^- \}.
$$
Next, we point out that 
\begin{equation*} \label{e:barta}
\begin{split}
     \big[ w(t_{\alpha+1}, x ) - w(t_{\alpha}, x )\big]^+ & \leq
      \big[ w(t_{\alpha+1}, x )- \lim_{s \to  s_{\alpha}^+ } w(s, x) \big]^++ 
    \big[  \lim_{s \to s_{\alpha}^+ } w(s, x) -
     \lim_{s \to s_{\alpha}^- } w(s, x) \big]^+ \\
 & \quad +
   [\lim_{s \to s_{\alpha}^- } w(s, x) - w(t_{\alpha}, x )\big]^+.
\end{split}
\end{equation*}
We now separately consider two cases. If $w(\cdot, x)$ is continuous at $s_\alpha$, then the second term in the above sum vanishes. If $w(\cdot, x)$ is not continuous at $s_\alpha$, then $s_\alpha$ must coincide with one of the discontinuity points $\tau_1, \dots, \tau_\ell$.   
We can then apply Lemma~\ref{l:upward} and owing to~\eqref{e:upward} we conclude that also in this case the second term in the above sum vanishes. This yields
\begin{equation*}
     \big[ w(t_{\alpha+1}, x ) - w(t_{\alpha}, x )\big]^+  \leq
      \big[ w(t_{\alpha+1}, x )- \lim_{s \to  s_{\alpha}^+ } w(s, x) \big]^++ 
   [\lim_{s \to s_{\alpha}^- } w(s, x) - w(t_{\alpha}, x )\big]^+
\end{equation*}
and we can control  the above terms by arguing as in {\sc Step 3}. This yields 
$$
    S_1 + S_2 + S_3 \leq C(\mathrm{Tot Var}\ w_0) 
$$
and owing to~\eqref{e:split} concludes the proof of~\eqref{e:tvm}. 
\end{proof}
\subsection{Proof of Theorem~\ref{l:tvflux}} \label{ss:proof} We proceed according to the following steps. \\
{\sc Step 1:} we show that~\eqref{e:tvm} implies~\eqref{e:tvflux2}.  This establishes~\eqref{e:tvflux2} provided $u(\cdot, x)$ has the same regularity as in the statement of Lemma~\ref{l:key}. 
To deduce~\eqref{e:tvflux2} from~\eqref{e:tvm} we recall that, if $u(\cdot, x)$ has the same regularity as in the statement of Lemma~\ref{l:key}, then 
 \be \label{e:deftv}
    \mathrm{Tot Var} \, w(\cdot, x) =   \sup_{0 \leq t_1 < t_2 < \dots< t_p \leq T} \sum_{\alpha=1}^{p-1} 
     | w(t_{\alpha+1}, x ) - w(t_{\alpha}, x ) |. 
\eq
Next, we point out that 
$$
   \sum_{\alpha=1}^{p-1} 
     [ w(t_{\alpha+1}, x ) - w(t_{\alpha}, x ) ]^+ =
   \sum_{\alpha=1}^{p-1} 
     [ w(t_{\alpha+1}, x ) - w(t_{\alpha}, x ) ]^- + w(t_p, x) - w(t_1, x)
$$
and by plugging the above expression into~\eqref{e:deftv} and using~\eqref{e:tvm} we arrive at
\be \label{e:carro}
        \mathrm{Tot Var} \, w(\cdot, x) \leq C(  \mathrm{Tot Var} \, w_0)
\eq 
{\sc Step 2:} we establish~\eqref{e:tvflux2} under the further assumption that $u$ is smooth ouside (a) a finite number of $C^1$ curves, the so-called \emph{shocks}, across which $u$ has a jump discontinuity.
{\color{black}At every point $(\tau, \gamma (\tau))$ belonging to the shock curve the assumptions of Lemma~\ref{l:upward} are satisfied}; (b) a finite number of points where two shocks interact (i.e., they intersect). {\color{black}We assume (a) and (b) and apply the Coarea Formula to each shock curve (or more precisely, to the $C^1$ function parameterizing each shock curve) and conclude that the hypotheses of Lemma~\ref{l:key} hold true for every $x \in \R \setminus N$, where $N$ is a negligible set. Owing to {\sc Step 1}, this implies that estimate~\eqref{e:carro} holds true for every $x \in \R \setminus N$.} Next, we recall Lemma~\ref{l:dafermos} and the fact that the total variation is lower semicontinuous with respect to weak$^\ast$ convergence: this implies that~\eqref{e:carro} holds true \emph{for every} $x \in \R$. \\
{\sc Step 3:} we conclude the proof of Theorem~\ref{l:tvflux}. \\
{\sc Step 3A:} we point out that, owing to the chain rule for $BV$ functions (see for instance~\cite[Theorem 3.96]{AmbrosioFuscoPallara}), the fact that~\eqref{e:carro} holds for every $x \in \R$ yields~\eqref{e:tvflux2}. \\
{\sc Step 3B:} by relying on a standard truncation and mollification argument, we construct a sequence $\{u_{0n} \}_{n \in \mathbb N} \subseteq C^\infty_c (\R)$ such that 
\be \label{e:fossile}
    u_{0n} \to u_0, \quad \mathrm{TotVar} \, u_{0n} \to \mathrm{TotVar} \, u_0 \quad 
    \text{as $n \to + \infty$}.
\eq
{\sc Step 3C:} we apply the Schaeffer Regularity Theorem~\cite{Schaeffer}. In particular, we apply the results by Dafermos~\cite{Daf} and we recall that, since the flux function $f$ satisfies $f'' <0$, then the entropy admissible solution $u$ of~\eqref{e:cl},~\eqref{e:cpdatum} does not have contact discontinuities. By using~\cite[\S3]{Daf} we conclude that from the sequence  $\{u_{0n} \}_{n \in \mathbb N}$ we can construct a second sequence $\{z_{0n} \}_{n \in \mathbb N} \subseteq C^\infty_c (\R)$ such that~\eqref{e:fossile} holds true and furthermore the entropy admissible solution $u_n$ of the Cauchy problem obtained by coupling~\eqref{e:fossile} with the condition $u(0, \cdot)= z_{0n}$ satisfies conditions (a) and (b) in {\sc Step 2}. Note that $u_n$ satisfies~\eqref{e:tvflux2}. \\
{\sc Step 3D:} we conclude the proof. We recall that the semigroup of entropy admissible solutions of~\eqref{e:cl},~\eqref{e:cpdatum} is $L^1$ stable with respect to the initial data, see~\cite[Formula (6.2.9)]{Dafermos:book}. We conclude that the sequence $u_n$ constructed in {\sc Step 3C}
converges to the entropy admissible solution $u$ of the Cauchy problem~\eqref{e:cl},~\eqref{e:cpdatum} in $L^1 (]0, T[ \times \R)$. This implies that, up to subsequences, for almost every $x \in \R$,  $u_n (\cdot, x)$ converges to $u(\cdot, x)$ in $L^1 (]0, T[)$ and hence, by the lower semicontinuity of the total variation with respect to the $L^1$ strong convergence, $w(\cdot, x)$ satisfies~\eqref{e:tvflux2}. Finally, we recall Lemma~\ref{l:dafermos} and the fact that the total variation is lower semicontinuous with respect to weak$^\ast$ convergence: this implies that~\eqref{e:tvflux2} is satisfied {for every} $x \in \R$.    \qed
\begin{remark}\label{r:iride}
By relying on the proof of Theorem~\ref{l:tvflux} one realizes that, if the initial datum $u_0$ is continuous, then one can control the left hand side of~\eqref{e:tvflux2} with $\mathrm{TotVar}\ w_0$. This is however not true in general: as a counterexample one can consider the Burgers' equation 
\be
\label{e:burgers}
      \partial_t u + \partial_x \left( u^2 \right)=0
\eq
and couple it with the Riemann-type initial datum 
\be \label{e:riemann}
    u_0 (x) : = \left\{ 
    \begin{array}{ll}
      -1 & x<0 \\
      1 & x>0. \\ 
    \end{array}
   \right.
\eq
Note that in this case $\mathrm{TotVar} \, w_0 =0$, however the solution of the Riemann problem~\eqref{e:burgers},~\eqref{e:riemann} is a rarefaction and~\eqref{e:tvflux2} fails. 
\end{remark}
\section{Further results on the time $BV$ regularity of entropy admissible solutions} \label{s:other}
\subsection{Proof of Proposition~\ref{p:borraccia2}}
In the following, just to fix the ideas, we assume $f' \ge 0$ on $[\operatorname{ess~inf} u_0, \operatorname{ess~sup} u_0]$, the proof in the case $f' \leq 0$ is analogous. 
We first establish~\eqref{e:borraccia2} in the case of wave front-tracking approximate solutions and then we pass to
the limit.
\subsubsection{Wave front-tracking approximation}
For the reader's convenience we briefly recall the construction of the wave front-tracking approximation. First,  we fix $\nu \in \mathbb N$ and we consider the conservation law 
\be \label{e:blu}
    \partial_t u^\nu + \partial_x [f^\nu(u^\nu)] = 0,
\eq 
where $f^\nu$ is the piecewise affine approximation of $f$ defined by interpolating the values of $f$ and setting 
\begin{equation*}
f^\nu(u) : = \frac{u-2^{-\nu}j}{2^{-\nu}}f(2^{-\nu}(j+1)) + \frac{2^{-\nu}(j+1) - u }{2^{-\nu}}f(2^{-\nu}j)
\quad \text{if $u \in [2^{-\nu}j, 2^{-\nu}(j+1)], j \in \mathbb{Z}$}.
\end{equation*} 
Next, we fix $u_0^\nu:\R \to 2^{-\nu}\mathbb{Z}$ with bounded variation and compact support and we assign the initial datum 
\be \label{e:rosso}
     u^\nu(0, \cdot) = u_0^\nu.
\eq
We then define the wave front-tracking approximate solution as the entropy admissible solution of the Cauchy problem~\eqref{e:blu},~\eqref{e:rosso}. Note that $u^\nu$ attains values in $2^{-\nu}\mathbb{Z}$ and that, for every $t>0$,  the function $u^\nu(t, \cdot)$ is piecewise constant. 
Note furthermore that the discontinuity points of $u^\nu$ are contained in the graphs of finitely many Lipschitz continuous curves (the so-called \emph{fronts}) $x_j$,  
$j = 1,\ldots, N$, and that the assumption $f'\ge 0$ yields $d x_j/dt  \ge 0$ for every $j=1,\ldots, N$.
We finally recall that there are only finitely many times at which a collision between two or more different fronts occurs. More precisely,  
we say that two fronts $x_i$ and $x_j$ collide at the time $\bar t$ if 
\begin{equation}
x_i(\bar t)= x_j(\bar t) \qquad \mbox{and} \qquad \exists \, \varepsilon >0  : x_i(t)\ne x_j(t) \quad \forall t \in ]\bar t - \varepsilon, \bar t[.
\end{equation}
We now establish~\eqref{e:borraccia2} in the case of wave front-tracking approximate solutions. 
\begin{lemma}\label{L_wft}
Fix $f \in C^2 (\R)$, $\nu \in \mathbb{N}$ and term $u^\nu$ the wave front-tracking approximate solution with initial datum $u_0^\nu$.
If $f'\ge 0$ on the interval $[\operatorname{ess~inf} u_0, \operatorname{ess~sup} u_0]$, then  
\begin{equation}\label{E_wft}
\mathrm{Tot Var} \ u^\nu(\cdot, x) \le \mathrm{Tot Var} \ u_0^\nu
\quad \text{for every $x \in \R$}.
\end{equation}
\end{lemma}
\begin{proof}
We denote as before by $x_j$, $j=1, \dots, n$, the wave fronts and we point out that
\begin{itemize}
\item[i)] since $d x_j/dt  \ge 0$ for every $j = 1,\ldots, N$, then for every but countably many $x \in \R$ there is a unique time $ t_j$ such that $x_j( t_j)= x$ ;
\item[ii)] by removing (if needed) a further finite set of $x$-s we can assume that no collision between different fronts occurs at time $t_j$;
\item[iii)] for the same $x$ as in point i) we have 
$u^\nu(t,  x^+)=u^\nu(t,  x^-)$ for every but finitely many $t \in ]0, + \infty[$.
\end{itemize}
Fix $\bar x$ satisfying properties i), ii) and iii) above and consider the function  $G^\nu_{\bar x} : ]0, + \infty[ \to \R$ defined by setting $G^\nu_{\bar x} (t) := G^\nu_{\bar x,1} (t)+ G^\nu_{\bar x,2} (t)$, where 
\begin{equation} \label{e:canarino}
G^\nu_{\bar x,1}(t)= \mathrm{TotVar}_{]-\infty,\bar x[} u^\nu(t,\cdot) \qquad \mbox{and} \qquad G^\nu_{\bar x,2}(t) =  \mathrm{TotVar}_{]0,t[} u^\nu(\cdot,\bar x) + |u^\nu(t^+, \bar x)- u^\nu(t^-, \bar x) |. 
\end{equation}
Assume for a moment that we have shown that $G^\nu_{\bar x}$ is a monotone non increasing function, then this yields~\eqref{E_wft} since
$$
\mathrm{TotVar} \ u^\nu(\cdot, \bar x) \le \lim_{t \to \infty} G^\nu_{\bar x}(t)\le \lim_{t \to 0^+} G^\nu_{\bar x}(t) = \mathrm{TotVar}_{]-\infty,\bar x[} \ u^\nu_0 \leq 
\mathrm{TotVar} \ u^\nu_0.
$$
We are thus left to show that $G^\nu_{\bar x}$ is a monotone non increasing function.
To this end, we point out that, by construction, the functions $G^\nu_{\bar x,1}, G^\nu_{\bar x, 2}$ and therefore  $G^\nu_{\bar x}$ are piecewise constant. Also, $G^\nu_{\bar x,1}$ is a monotone non increasing function: more precisely, it can only diminish at times when an interaction between wave fronts occurs on the interval $]-\infty, \bar x[$. Since discontinuities of  $G^\nu_{\bar x,2}$ can only occur at a point $t_j$ as in item i) above, to conclude it suffices to show that 
for every $j=1,\ldots, N$ we have $G^\nu_{\bar x}(t_j^-)=G^\nu_{\bar x}(t_j^+)$. To this end, we point out that
\begin{equation}
\begin{split}
G^\nu_{\bar x,1}(t_j^+)  - G^\nu_{\bar x,1}(t_j^-) = &-  |u^\nu(t_j,
\bar x^+)- u^\nu(t_j,\bar x^-)|  \\
G^\nu_{\bar x,2}(t_j^+)  - G^\nu_{\bar x,2}(t_j^-) = & |u^\nu(t_j,\bar x^+)- u^\nu(t_j,\bar x^-)|.  \\
\end{split}
\end{equation}
In particular $G^\nu_{\bar x}(t_j^-)= G^\nu_{\bar x,1}(t_j^-) +G^\nu_{\bar x,2}(t_j^-) = G^\nu_{\bar x,1}(t_j^+) +G^\nu_{\bar x,2}(t_j^+) = G^\nu_{\bar x}(t_j^+)$. This establishes~\eqref{E_wft} for every $ \bar x \in E$ for some suitable set $E$ such that $\mathcal L^1(\R\setminus E)=0$. We are left to show that actually~\eqref{E_wft} holds for every $x \in \R$: to this end, we recall Lemma \ref{L_trace} and the lower
semicontinuity of the total variation with respect to the $L^1$-convergence and we conclude that 
\begin{equation}\label{E_TV_trace}
\begin{split}
\mathrm{TotVar} \ u^\nu(\cdot,  x^+) & \le \liminf_{ x_n \uparrow  x}\mathrm{TotVar}\ u^\nu(\cdot,  x_n) \leq \mathrm{TotVar} \ u_0, \\
\mathrm{TotVar} \ u^\nu(\cdot,  x^-)& \le \liminf_{ y_n \downarrow  x}\mathrm{TotVar} \ u^\nu(\cdot,  y_n) \leq \mathrm{TotVar} \ u_0, 
\end{split}
\end{equation}
for suitable sequences $\{ x_n\}_{n \in \mathbb{N}}, \{y_n\}_{n \in \mathbb{N}} \subseteq E$.  
\end{proof}

\begin{proof}[Proof of Proposition~\ref{p:borraccia2}]
Given $u_0 \in BV(\R)$  we fix a family $\{ u_0^\nu\}  \subseteq BV(\R)$ with compact support attaining values in $2^{-\nu}\mathbb{Z}$ such that 
$u_0^\nu \to u_0$ in $L^1(\R)$ and $\mathrm{TotVar}\ u_0^\nu \to \mathrm{TotVar} \ u_0$ as $\nu \to 0^+$ (see \cite[Lemma 2.2]{Bressan_book}).
Let $u^\nu$ be the corresponding family of wave front-tracking approximate solutions with initial datum $u_0^\nu$. By the analysis in~\cite[Chapter 6]{Bressan_book} we infer that $u^\nu \to u$ in $L^1_{\mathrm{loc}}(]0, + \infty[ \times \R)$. This implies that 
$
u^\nu (\cdot, x) \to u(\cdot,x)$  in $L^1_{\mathrm{loc}}(]0, + \infty[) $
for a.e. $x \in \R$. Since for a.e. $x \in \R$ we have $u(\cdot, x^+)=u(\cdot, x^-)$, then by combining 
Lemma \ref{L_wft} with the lower semicontinuity of the total variation we get 
\begin{equation}
\mathrm{TotVar} \ u(\cdot,x) \le \liminf_{\nu \to \infty} \mathrm{TotVar}\ u^\nu(\cdot,x) \stackrel{\eqref{E_wft}}{\le} \liminf_{\nu \to \infty}  \mathrm{TotVar}\ u^\nu_0 =  \mathrm{TotVar} \ u_0 \quad \text{for a.e. $x \in \R$} 
\end{equation}
By using the same approximation argument as in the proof  of Lemma \ref{L_wft} we conclude that the above estimate holds for every $x \in \R$. 
\end{proof}
\subsection{Initial-boundary value problems}
Theorem~\ref{l:tvflux} and Proposition~\ref{p:borraccia2} have several extensions to initial-boundary value problems. Here we only explicitely discuss the extension we need in the proof of Theorem~\ref{t:propbvreg} and of Corollary~\ref{c:stability}. 
\begin{corol} \label{c:stazione}
Fix $f \in C^2 (\R)$,  $\bar u \in BV (]0, T[)$, $u_0 \in BV (]\alpha, \beta[)$. Assume furthermore that
$f' \ge 0$ on the interval $[\min \{ \inf \bar u, \inf u_0 \}, \max \{ \sup \bar u, \sup u_0 \} ]$. 
Let $u$ be the unique entropy admissible solution of the initial-boundary value problem~\eqref{e:cl}, \eqref{e:data2}. Then 
\be \label{e:stazione}
     \mathrm{TotVar} \ u(\cdot, x^\pm) \leq     
   \mathrm{TotVar} \ u_0 +   \mathrm{TotVar} \ \bar u + |u_0 (\alpha^+) - \bar u (0^+)|, \quad \text{for every $x \in ]\alpha, \beta[$}.
\eq
In the above expression, we denote by $u_0 (\alpha^+)$ and $\bar u (0^+)$ the right limit of $u_0$ and $\bar u$ at $x=\alpha$ and $t=0$, respectively.\footnote{These limits exist owing to the $BV$ regularity of $u_0$ and $\bar u$.}
\end{corol}
Under the same assumptions as in Corollary~\ref{c:stazione}, the trace at $x = \beta$ of the entropy admissible solution of the initial-boundary value problem~\eqref{e:cl}, \eqref{e:data2} is well defined and we denote it by $u_\beta$.  By using the lower semicontinuity of the total variation with respect to the strong convergence we conclude that 
\be 
\label{e:treno}
         \mathrm{TotVar} \ u_\beta  \leq 
     \mathrm{TotVar} \ u_0 +   \mathrm{TotVar} \ \bar u + |u_0 (\alpha^+) - \bar u (0^+)|. 
\eq
\begin{remark}
Even if we do not explicitely discuss it, Theorem~\ref{l:tvflux} extends to initial-boundary value problems and provides a control of $\mathrm{TotVar} \, w(\cdot, x)$ for every $x \in ]\alpha, \beta[$. As in the case of the Cauchy problem, if $f'$ changes sign we cannot hope for a control on the total variation of $u$: for a counterexample, we refer to the construction detailed in \S\ref{ss:tvblowup}, which also applies to initial-boundary value problems. 
\end{remark}
\begin{proof}[Proof of Corollary~\ref{c:stazione}]
We only provide a sketch of the proof, which is based on the same argument as the proof of Proposition~\ref{p:borraccia2}. The key point is the construction of the wave front-tracking approximation of the initial-boundary value problem~\eqref{e:cl}, \eqref{e:data2}. We fix  
$\nu \in \mathbb N$ and $\bar u^\nu: ]0 , T[ \to 2^{-\nu} \mathbb Z$,  we assign 
the boundary datum 
\be \label{e:verde}
     u^\nu (\cdot, \alpha)= \bar u^\nu 
\eq
and we construct the entropy admissible solution of the initial-boundary value problem~\eqref{e:blu},\eqref{e:rosso},\eqref{e:verde}. To construct the wave front-tracking approximation, the main difference with respect to the 
Cauchy problem is that we have to define the solution of the initial-boundary value problem 
obtained by coupling~\eqref{e:blu} with the data 
\be \label{e:viola}
   u^\nu (\cdot, \alpha) = u_b, \quad u^\nu (0, \cdot) = u_i,
\eq
where $u_b, u_i \in 2^{-\nu} \mathbb Z$ satisfy $(f^\nu)'(u_b^\pm) \ge 0$, 
$(f^\nu)'(u_i^\pm)\ge 0$. We term $z$ the entropy admissible solution of the Cauchy problem 
obtained by coupling~\eqref{e:blu} with the initial datum 
$$
    u^\nu (0, x) :  =
    \left\{ 
            \begin{array}{ll}
                u_b & x < \alpha \\
                u_i  & x >\alpha. \\         
            \end{array}
    \right. 
$$
We claim that the restriction of $z$ to $]\alpha, \beta[$ is the entropy admissible solution of the initial-boundary problem~\eqref{e:blu},\eqref{e:viola}. To see this we have to 
verify~\eqref{e:bln2}. First, we point out that $z$ satisfies the
entropy inequality inside the domain $[0, T[ \times ]\alpha, \beta[$, namely
\be \label{e:rosa}
 \int_0^T \int_\alpha^\beta |z-c| \partial_t \phi + \mathrm{sign} (z-c) [f^\nu(z) - f^\nu(c)] \partial_x \phi dx dt + \int_{\alpha}^{\beta} \phi (0, \cdot) | u_i - c| dx \ge 0
\eq
for every $c \in \R$ and every $\phi \in C^\infty_c (]-\infty, T[ \times ]\alpha, \beta[)$ such that $\phi \ge 0$. Next, we claim that 
\be \label{e:indaco}
      \big[  
       \mathrm{sign} (u_b -c)-
     \mathrm{sign} (z_\alpha-c) \big]
      \big[ f^\nu(z_\alpha) - f^\nu (c)\big] = 0, \quad \text{a.e. on $]0, T[$.}
\eq
To see this, we recall that, since $(f^\nu)'\ge 0$, then $f^\nu(z_\alpha) = f^\nu (u_b)$ and by a case-by-case analysis we arrive at~\eqref{e:indaco}. Next, we fix a parameter $\ee >0$ and set $\eta_\ee: = 1- \omega_\ee$, where $\omega_\ee \in C^\infty (\R)$ is the same family as in~\eqref{e:cremisi2}.
We then fix a test function $\varphi \in C^\infty_c (]-\infty, T[\times ]- \infty, \beta[)$, plug the test function $\phi_\ee (t, x): = \varphi (t, x)\eta_\ee (x)$ into~\eqref{e:rosa} and let $\ee \to 0^+$. By using the fact that 
$$
     \int_0^T \! \! \! \int_\alpha^\beta \mathrm{sign} (z-c) [f^\nu(z) - f^\nu(c)]  
     \eta_\ee' \varphi  dx dt  \! \to \! 
     \int_0^T \mathrm{sign} (z_\alpha-c) [f^\nu(z_\alpha) - f^\nu(c)]  
      \varphi (\cdot, \alpha) dt \quad \text{as $\ee \to 0^+$}
$$
and recalling~\eqref{e:indaco} we arrive at~\eqref{e:bln2}.

Once we have constructed the wave front-tracking approximation, the proof follows the same argument as the proof of Proposition~\ref{p:borraccia2}, with the only difference that the term $G^\nu_{\bar x, 1}$ in~\eqref{e:canarino} should be replaced by
\begin{equation*}
    G^\nu_{\bar x, 3} (t): = \mathrm{TotVar}_{]\alpha,\bar x[} \ u^\nu(t,\cdot) + 
    |\bar u^\nu (t^+) - u^\nu (t, \alpha^+) |+
    \mathrm{TotVar}_{]t, + \infty[} \ \bar u^\nu.
    \qedhere
\end{equation*} 
\end{proof}

\subsection{An example of total variation blow up for $u(\cdot, 0)$} \label{ss:tvblowup}
In this paragraph we exhibit a counterexample showing that, if $f'(u)$ can change sign, then the total variation of the entropy admissible solution $u(\cdot, 0)$ can blow up in finite time, even if the initial data (in the case of the Cauchy problem) or the initial and boundary data (in the case of the initial-boundary value problem) have bounded total variation. 
\subsubsection{Construction roadmap}
We consider the Burgers' equation~\eqref{e:burgers}
and we first provide an heuristic discussion of the basic ideas underpinning the contruction of the counterexample. The key point in the analysis is the construction of a map $\gamma: ]0, T[ \to \R$, $T>0$ to be determined in the following, which exhibits the following features {\color{black}(see Figure~\ref{F_TVu0} for a representation)}:
\begin{itemize}
\item the curve $\gamma$ crosses the vertical axis $x=0$ infinitely many times;
\item  there is a function $u_-: ]0, T[ \to \R$ such that i) $u_-(t) \sim 1$; ii) $u_-(t) $ and $u_+(t) =-1$ satisfy the Rankine-Hugoniot conditions 
\be \label{e:ex:rh}
    u_-(t)^2 - u_+(t)^2 = \gamma' (t) [u_-(t) - u_+(t)] \iff
    \gamma'(t) =  u_-(t) -1. 
\eq
Since in the case of convex fluxes the Lax entropy admissible conditions boil down to the inequality $u_-(t) \ge u_+(t)$, the equality~\eqref{e:ex:rh} dictates that $\gamma$ is an (entropy admissible) shock curve between $u_-(t)$ (on the left) and $-1$ (on the right).  
\end{itemize}
Assume for a moment that $\gamma$ is indeed the shock curve of the solution $u$ of a Cauchy problem, then the total variation of $u(\cdot, 0)$ must blow up: indeed, $u$ is close to $1$ on the left of $\gamma$, and equal to $-1$ on the right. Since $\gamma$ crosses the vertical axis $x=0$ infinitely many times, then $u(\cdot, 0)$ oscillates infinitely many times between a value close to $1$ and the value $-1$, and hence its total variation must blow up. To construct the initial datum $u_0$ of this Cauchy problem, we proceed as follows. First, we recall that, in the subsets of the $(t, x)$ plane where $u$ is regular, $u$ is constant along the characteristic lines, which have speed $2u$. Next, we consider the line $\xi_t$ with slope $2u_-(t)$ and passing through the point $(t, \gamma(t))$ and we define it on the interval $[0, t]$ since we want to focus on backward characteristics. Since $u_-(t) =1 + \gamma'(t)$ owing to~\eqref{e:ex:rh}, then the backward characteristic passing through $(t, \gamma(t)^-)$ is
\be \label{e:ex:backchar}
      \xi_t (s)= 2 [1+  \gamma'(t) ]s + \gamma (t) -2 [1+  \gamma'(t) ]t, \quad 
     s \in [0, t]. 
\eq
By enforcing suitable conditions on $\gamma$, we get that, if $t_1 \neq t_2$, then $\xi_{t_1}$ and $\xi_{t_2}$ do not cross. In this way we can ``pull back" the values of $u_-(t)$ to the initial time and define the initial datum $u_0$ in such a way that $u_0 (\xi_t (0)) = u_- (t)$. We can then easily enforce the condition  $\mathrm{Tot Var} \, u_0 <+ \infty$, and extend the construction to define an initial-boundary value problem. 
\begin{figure}
\centering
\def\svgwidth{0.7\columnwidth}
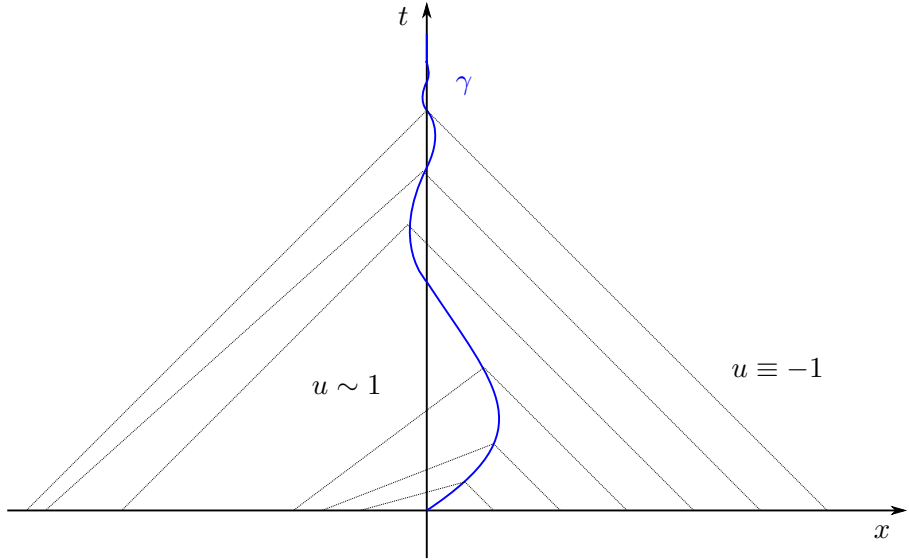
\caption{A solution $u$ of the Burgers equation with $\mathrm{Tot Var}\, u(\cdot,0) = +\infty$.} \label{F_TVu0}. 
\end{figure}
\subsubsection{Technical details}
We now provide the detailed construction of the counterexample, which is achieved in several steps. \\
{\sc Step 1:} construction of the ``building blocks" 
of the curve $\gamma$. The curve $\gamma$ is constructed by alternatively patching together suitably rescaled ``right curved " and ``left curved" blocks. We first construct the {\color{black}right} curved building block. We fix a parameter 
$\ee <1$ and we set  
\be \label{e:ex:rc}
    \hat \gamma_\ee: [0, \ee] \to \R, 
   \qquad \hat \gamma_\ee (t) = \frac{1}{2 \ee} t^3 - \frac{3}{2} t^2 + \ee t.
\eq
Note that 
\be \label{e:ex:dati}
    \hat \gamma_\ee(0) =0, \quad \hat \gamma_\ee (\ee) =0, \quad \hat \gamma_\ee'(0) = \ee, \quad \hat \gamma_\ee'(\ee) =- \frac{\ee}{2}, 
  \quad 
  \hat \gamma_\ee  \ge 0, \quad \hat \gamma_\ee'' \leq 0  \, \text{on $[0, \ee]$.} 
\eq
{\sc Step 2:} analysis of backward characteristics.
We recall~\eqref{e:ex:backchar} and we point out that 
\be \label{e:ex:derflux}
        \frac{\partial \xi_t}{\partial t} (s)= 2 \gamma''(t) [s-t] - \gamma' (t) - 2,
\eq
which yields 
\be \label{e:ex:segno} 
     \gamma''(t) < 0 \implies \frac{\partial \xi_t}{\partial t} (s)<0 \, \text{for $s \in ]s_0(t), t]$, where} \; s_0 (t)= t + \frac{\gamma'(t) + 2}{2 \gamma''(t)}.
\eq
Let us focus on the case $\gamma= \hat \gamma_\ee$: we have 
$\gamma''(t) =0$ if and only if $t = \ee$ and $\gamma''(t)<0$ if $t \in [0, \ee[$. If $t=\ee$, then $\partial \xi_t /\partial t = - \hat \gamma_\ee '(\ee) - 2<0$ for every $s \in \R$. If $t \in [0, \ee[$, then 
$$
   s_0 (t)= t + \frac{\hat \gamma_\ee'(t) + 2}{2 \hat \gamma_\ee''(t)} 
   \leq 
   t + \sup_{t \in [0, \ee[}  
   \frac{\hat \gamma_\ee'(t) + 2}{2 \hat \gamma_\ee''(t)} 
    \stackrel{\hat \gamma''_\ee \leq 0,  \; \hat \gamma_\ee' >-2 }{\leq}
   t+ \frac{ \inf_{t \in [0, \ee[} \hat \gamma_\ee'(t) + 2 }{ \inf_{t \in [0, \ee[}
     2 \hat \gamma_\ee''(t) } \stackrel{\eqref{e:ex:rc}}{\leq}
   t+  \frac{1}{-6}     = t- \frac{1}{6}.
$$
Summing up, we conclude that 
\be 
\label{e:ex:nointerse}
      \frac{\partial \xi_t}{\partial t} <0
       \quad 
      \text{for every $t \in [0, \ee]$ and} \;  s > t - \frac{1}{6}. 
\eq
Finally, we define the {\color{black}left} curved building block as $-\hat \gamma_\ee(t)$. Note that the speed of the minimal backward characteristic through $(t, - \hat \gamma_\ee (t)^-)$ is $2 [1 - \hat \gamma_\ee']$ and, since $\hat \gamma''_\ee \leq 0$, then the backward characteristics do not intersect, namely 
\be 
\label{e:ex:nointerse2}
     \frac{\partial \xi_t}{\partial t} <0 \quad \text{for every $t \in [0, \ee]$ and $s \leq t$}. 
\eq
{\sc Step 3:} we define the shock curve $\gamma$. We set 
\be \label{e:ex:gamma}
    \gamma (t) : = \sum_{n=3}^\infty 
    (-1)^n  \hat \gamma_{\ee_n} (t - \tau_n)\mathbbm{1}_{I_n}(t),
   \quad \ee_n:=2^{-(n+1)},
    \; I_n : = [2^{-3} - 2^{-n}, 2^{-3}- 2^{-(n+1)}[, 
   \; \tau_n: = 2^{-3} - 2^{-n}. 
\eq 
In the above expression, $\mathbbm{1}_{I_n}$ denotes the characteristic function of the interval $I_n$. Note that the interval $I_n$ are disjoints and hence the above series converges since it is locally finite. Note furthermore that $\gamma$ is obtained by patching together infinitely many $C^\infty$ arcs and that at the junction points both the functions and its first derivatives match, hence $\gamma \in C^1 (]0, 2^{-3}[)$. 
We now consider the backward characteristics with final point $(t, \gamma(t)^-)$, we recall~\eqref{e:ex:backchar},~\eqref{e:ex:nointerse} and~\eqref{e:ex:nointerse2}. Since $2^{-3} < 1/6$, we conclude that the map $t \mapsto \xi_t(0)$ is strictly decreasing (and henceforth injective) on $[0, 2^{-3}[$. Owing to~\eqref{e:ex:backchar} and to the fact that $\gamma \in C^1$, it is also continuous, and hence the image of the interval $[0, 2^{-3}[$ is an interval, which we term $]-r, 0]$. The exact expression of the number $r>0$ could be explicitly computed but is not relevant here. We term $\varphi:]-r, 0] \to [0, 2^{-3}[$ the inverse of the map $t \mapsto \xi_t(0)$.  We can now define the initial datum $u_0: \R \to \R$ by setting 
\be \label{e:ex:id}
     u_0(x) : =
     \left\{
     \begin{array}{ll}
               0 & x < -r \\
               1 + \gamma'(\varphi(x)) & - r < x < 0  \\
              -1 &  x>0. \\
     \end{array}
     \right.
\eq
Note that 
\be 
\label{e:ex:tvu0}
     \mathrm{Tot  Var} \, u_0 \leq 5 +   \sum_{n=3}^\infty \int_{I_n} |\gamma'' (t)| dt
    =  5 + \sum_{n=3}^\infty 2^{-(n+2)} < + \infty.
\eq
{\sc Step 4:} we show that the entropy admissible solution of the Cauchy problem obtained by coupling~\eqref{e:burgers} with~\eqref{e:ex:id} satisfies $ \mathrm{Tot  Var}  \ u(\cdot, 0) = + \infty$. The entropy admissible solution $u$ is given by 
\be \label{e:ex:u}
    u(t, x) = 
    \left\{
    \begin{array}{ll}
    0 & x < \lim_{y \to -r^+} \xi_{\varphi(y)} (t) \\
    1 + \gamma' (\varphi(y)) & x = \xi_{\varphi(y)} (t), \; x < \gamma(t) \\
     -1 & x  > \gamma (t). \\
   \end{array}
    \right.
\eq
In other words, $u$ is identically equal to $-1$ for $x> \gamma(t)$, and it is transported along the characteristic lines~\eqref{e:ex:backchar} for $x < \gamma(t)$. Note that by construction the characteristics do not intersect on the set $x < \gamma (t)$. We then get 
\be \label{e:ex:dalbasso}
   \mathrm{TotVar} \ u(\cdot, 0) \ge \sum_{n=3}^\infty |u(\sigma_{n+1}, 0) - u (\sigma_n, 0)|, \quad \sigma_n = 2^{-3} -\frac{3}{2}  2^{-(n+1)}.
\eq
Note that $\sigma_n$ is the middle point of the interval $I_n$, and that $(\sigma_n, 0)$ is a continuity point for $u$. Note furthermore that if $n$ is odd then $u(\sigma_n, 0)= -1$. If $n$ is even, $u(\sigma_n, 0) = 1 + \gamma'(\varphi(x))$, for some $x \in ]-r, 0]$. Since $1 + \gamma'(\varphi(x))>0$, then by using~\eqref{e:ex:dalbasso} we conclude that
$$
 \mathrm{TotVar} \ u(\cdot, 0) \ge \sum_{n=3, \; n \; \text{odd}}^\infty 
   1 = + \infty
$$
and this concludes the analysis of the Cauchy problem. \\
{\sc Step 5:} extension to the initial-boundary value problem. 
By restricting the function $u$ in~\eqref{e:ex:u} to the strip $]-r-1, 1[$
we get a solution of the initial boundary value problem obtained by coupling~\eqref{e:burgers} with the initial datum~\eqref{e:ex:id} and the boundary data $u(t, -r-1) \equiv 0$, $u(t, 1) \equiv -1$.
\section{The multi-path model: distributional formulation, existence and uniqueness results}\label{s:normal}
\subsection{Distributional formulation of the initial-boundary value problem~\eqref{e:te},~\eqref{e:id2} and~\eqref{e:bd2}}
We now complete the definition of distributional solution of the multi-path model. 
We first need some preliminary remarks:  fix $r_k \in L^\infty (]0, T[ \times P_k)$ and assume that $\theta_k \in L^\infty (]0, T[ \times P_k)$ satisfies 
\be \label{e:marzo}
        \int_0^T \! \! \int_\alpha^\beta r_k \theta_k  (\partial_t \phi +  v(r_k)  \partial_x \phi ) 
      dx dt     = 0
    \quad \text{for every $\phi \in C^\infty_c (]0, T[ \times P_k)$}.
\eq
By applying Lemma~\ref{l:normaltrace2} with $b=v (r_k)$ and $]\alpha, \beta[= P_k$  we define the initial value $[r_k \theta_k]_0$ and the distributional trace 
$\mathrm{Tr}[v(r_k)r_k \theta_k] (\cdot, a^+)$, where we recall that $a$ is the starting point of the path $P_k$. Also, assume that $r_k$ is obtained by patching together the $\rho_i$-s as in~\eqref{e:patches} and that each $\rho_i$ is an entropy admissible solution of~\eqref{e:foglio}, that is it satisfies~\eqref{e:gennaio} with $f(\rho_i) = v(\rho_i) \rho_i$. Then owing to Theorem~\ref{l:verde} the trace $v(r_k) r_k(\cdot, a^+)$ is attained as a strong limit in the $L^1$ topology. 
\begin{definition}
\label{d:te} Assume that $r_k$ satisfies~\eqref{e:patches}, where each $\rho_i$ is an entropy admissible solution of~\eqref{e:foglio},~\eqref{e:id1}.  We term $\theta_k \in  L^\infty (\R_+ \times P_k)$ a distributional solution of the initial-boundary value problem~\eqref{e:te},~\eqref{e:id2} and~\eqref{e:bd2} if $\theta_k$ satisfies~\eqref{e:dsoltheta} and furthermore 
\be \label{e:sasso}
       \mathrm{Tr} [ v(r_k) r_k \theta_k](\cdot, a^+) = \bar \theta_k  v(r_k)r_k (\cdot, a^+), \qquad [r_k \theta_k]_0 = \theta_{k0} \rho_{i0} \; \text{on $I_i$, for every $i$ such that $I_i
\subseteq P_k$} .
\eq
\end{definition}
We have 
\begin{lemma}
\label{l:iff}
 Assume that $r_k$ satisfies~\eqref{e:patches}, where each $\rho_i$ is an entropy admissible solution of~\eqref{e:foglio},~\eqref{e:id1}. Assume furthermore that  $\theta_k \in L^\infty (]0, T[ \times P_k)$ is a distributional solution of~\eqref{e:te}, \eqref{e:id2},~\eqref{e:bd2}. Under~\eqref{e:sommaunoid},~\eqref{e:sommaunobd}, the following conditions are equivalent:
\begin{itemize}
\item[i)] equation~\eqref{e:sum1} is satisfied a.e. on $]0, T[ \times I_i$, for every $i$;
\item[ii)] we have the equality 
\be \label{e:fragola}
     g( \rho_{j+1}) (\cdot, d^+) \stackrel{\eqref{e:f}}{=} v(\rho_{j+1}) \rho_{j+1} (\cdot, d^+) =
      \sum_{k: I_{j+1} \subseteq P_k }\mathrm{Tr} [v(\rho_{j}) \rho_j \theta_k](\cdot, d^-)
\eq
at every junction point $d$.  In the above expression $d$ is the final point of the road $I_j$ entering the junction and the starting point of the road $I_{j+1}$ exiting the junction. 
\end{itemize}
\end{lemma}
\begin{proof}[Proof of Lemma~\ref{l:iff}] 
{\sc Step 1:} we establish the implication i)$\implies$ii). We apply~\eqref{e:sum1} on the road $I_{j+1}$ and obtain the second  of the following equalities: 
\begin{equation*} 
\begin{split}
     v(\rho_{j+1}) \rho_{j+1} (\cdot, d^+) &
    \stackrel{\eqref{e:maggio}}{=}  -
    \mathrm{Tr}[v(\rho_{j+1}) \rho_{j+1} ](\cdot, d^+) \stackrel{\eqref{e:sum1}}{=}
    - \sum_{k: I_{j+1} \subseteq P_k }\mathrm{Tr} [v(r_k) r_k \theta_k](\cdot, d^+) \\
    &
   \stackrel{\eqref{e:mirtillo}}{=}
     \sum_{k: I_{j+1} \subseteq P_k }\mathrm{Tr} [v(r_k) r_k \theta_k](\cdot, d^-) 
\end{split}
\end{equation*}
and owing to~\eqref{e:patches} this yields~\eqref{e:fragola}. \\
{\sc Step 2:} we establish the implication ii)$\implies$i). We argue by induction. First, we show that~\eqref{e:sum1} holds on $]0, T[ \times I_1$.
To this end, we set $z: = \sum_{k=1}^m \theta_k$ and we point out that, owing to~\eqref{e:sommaunoid} and~\eqref{e:sommaunobd} and to the linearity of the equation for $\theta_k$, $z$ is a solution of the initial-boundary value problem
\be \label{e:lampone}
    \left\{
    \begin{array}{ll}
             \partial_t [ \rho_1 z ] + \partial_x [v(\rho_1) \rho_1 z] =0 \\
             z (0, \cdot) =1, \qquad z(\cdot, a) = 1,
    \end{array} 
    \right.
\eq 
where we recall that $a$ is the initial point of the interval $I_1$. Owing to~\eqref{e:foglio}, $z \equiv 1$ is a solution of the above initial-boundary value problem. We can then apply Theorem~\ref{t:luglio} with $]\alpha, \beta[=I_1$, $b = v(\rho_1)$ and by the uniqueness part we conclude that~\eqref{e:sum1} holds true on $]0, T[ \times I_1$.

Next, we fix $i=2, \dots, n_k$, assume that ~\eqref{e:sum1} holds true on $]0, T[ \times I_j$ for $j=1, \dots, i-1$ and, under~\eqref{e:fragola}, show that it holds true on $]0, T[ \times I_i$. To this end, we term $d$ the junction point between the road $I_{i-1}$ and the road $I_i$, that is $d$ is the final point of the road $I_{i-1}$ and the initial point of the road $I_i$.  We recall~\eqref{e:patches} and point out that 
\begin{equation*}\begin{split}
            \mathrm{Tr} 
           \left[  \sum_{k: I_i \subseteq P_k} v(\rho_i) \rho_i \theta_k  \right]    (\cdot, d^+) &=
 \sum_{k: I_i \subseteq P_k} \mathrm{Tr}[v(\rho_i) \rho_i \theta_k] (\cdot, d^+) \stackrel{\eqref{e:mirtillo}}{=} - 
     \sum_{k: I_i \subseteq P_k} \mathrm{Tr}[v(\rho_{i-1}) \rho_{i-1} \theta_k] (\cdot, d^-) \\ &
     \stackrel{\eqref{e:maggio},\eqref{e:fragola}}{=}
     \mathrm{Tr}[v(\rho_{i}) \rho_{i} ](\cdot, d^+) 
\end{split}
\end{equation*}
and by recalling~\cite[Definition 2.7]{DovettaMarconiSpinolo} this implies that the function $
\sum_{k: I_i \subseteq P_k} v(\rho_i) \rho_i \theta_k $ attains the boundary condition $1$. We can then repeat the same argument as before on the interval $I_i$ and conclude that~\eqref{e:sum1} holds true on $]0, T[ \times I_i$.
\end{proof}
\subsection{Proof of Theorem~\ref{t:exunisd}}
We first establish existence, next uniqueness.  
\subsubsection{Existence} \label{sss:exsd}
We fix a path $P_k$ and we term $I_1, \dots, I_{n_k}$ the consecutive roads composing the path $P_k$. We now construct the solutions $\rho_1, \dots, \rho_{n_k}$ and $\theta_k$. We argue inductively: first, we construct the solution on $I_1$. Next, we assume that we have constructed a solution on the road $I_1, \dots, I_j$ and we construct it on $I_{j+1}$. \\
{\sc Construction of the solution on $I_1$.} To construct $\rho_1$, we apply Proposition~\ref{p:datainc} with $\bar u= \bar \rho$, $u_0 = \rho_{10}$ and $f(u):= u v(u)$. We conclude that there is an entropy admissible solution of~\eqref{e:foglio}, \eqref{e:id1}, \eqref{e:bd1} such that $0 \leq \rho \leq \rho^\ast$ and we recall that, owing to~\eqref{e:f}, the point $\rho^\ast$ is the point where the function $u \mapsto v(u) u$ attains its maximum. Next, we apply Theorem~\ref{t:luglio} with $]\alpha, \beta[=I_1$ and $b= v(\rho_1)$ and conclude that there is a solution of~\eqref{e:te}, \eqref{e:id2}, \eqref{e:bd2} defined on $I_1$. Since $0 \leq \bar \theta_k \leq 1$ and $0 \leq \theta_{k0} \leq 1$, then by the comparison principle given in Theorem~\ref{t:luglio}  
we get~\eqref{e:zerouno} on $]0, T[ \times I_1$.  To conclude the existence proof on $I_1$, we are left to establish~\eqref{e:sum1} on $]0, T[ \times I_1$. To this end, we can argue as in {\sc Step 2} of the proof of Lemma~\ref{l:iff}. \\
{\sc Inductive step.} We assume that we have constructed the solution on $I_1, \dots, I_j$ and we construct it on $I_{j+1}$. More precisely, we assume that we have constructed the functions $\rho_1, \dots, \rho_j$ and the function $\theta_k$ on $I_1, \dots, I_j$. We also assume that $0 \leq \rho_i \leq \rho^\ast$, for every $i=1, \dots, j$, and that~\eqref{e:zerouno} and~\eqref{e:sum1} are both satisfied on $]0, T[ \times I_i$, for every $i=1, \dots, j$. 
We term $d$ the junction point, that is $d$ is the final point of the road $I_j$ and the initial point of the road $I_{j+1}$. We proceed according to the following steps. \\
{\sc Step 1:} we show that 
\be 
\label{e:rangeok} 
      0 \leq \sum_{k: I_{j+1} \subseteq P_k} \mathrm{Tr}[v (\rho_{j}) \rho_j \theta_k] (\cdot, d^-) \leq v(\rho^\ast) \rho^\ast \quad \text{a.e. on $]0, T[$}. 
\eq
To establish~\eqref{e:rangeok} we recall that, by the inductive assumption,~\eqref{e:zerouno} is satisfied on $]0, T[ \times I_j$. Also, owing to the specific structure of the network, 
$$
    \{ k: I_{j+1} \subseteq P_k \} \subseteq \{ k: I_{j} \subseteq P_k \}
$$
and owing to the inequality $v \ge 0$ this yields 
$$
    0 \leq v(\rho_j) \rho_j  \sum_{k: I_{j+1} \subseteq P_k}   \theta_k
     \leq v(\rho_j) \rho_j \sum_{k: I_{j} \subseteq P_k}   \theta_k \stackrel{\eqref{e:sum1}}{=} v(\rho_j) \rho_j 
    \stackrel{\eqref{e:f}}{\leq}    v(\rho^\ast) \rho^\ast.
$$ 
By a small modification of the proof of~\cite[Lemma 6.1]{DovettaMarconiSpinolo} one can show that the above inequalities yield~\eqref{e:rangeok}. \\
{\sc Step 2:}  we construct the function $\rho_{j+1}$. We combine~\eqref{e:f} and~\eqref{e:rangeok} and we conclude that there is a unique function $\bar \rho_{j+1} \in L^\infty (]0, T[)$ such that 
\be \label{e:barrho}
    0 \leq \bar \rho_{j+1} \leq \rho^\ast, \qquad 
    v(\bar \rho_{j+1}) \bar \rho_{j+1} \stackrel{\eqref{e:f}}{=} g(\bar \rho_{j+1}) =  \sum_{k: I_{j+1} \subseteq P_k} \mathrm{Tr}[v (\rho_{j}) \rho_j \theta_k] (\cdot, d^-), \quad 
    \text{a.e. on $]0, T[$.}
\eq
Next, we apply Proposition~\ref{p:datainc} with $]\alpha, \beta[=I_{j+1}$, 
$\bar u = \bar \rho_{j+1}$ and we term $\rho_{j+1}$ the entropy admissible solution such that $0 \leq \rho_{j+1} \leq \rho^\ast$.
\\
{\sc Step 3:} we define the function $\theta_k$ on $I_{j+1}$. By combining the fact that~\eqref{e:zerouno} is satisfied on $I_j$ with the inequality $v \ge 0$ and recalling that by assumption $I_{j+1} \subseteq P_k$ we get 
$$
    0 \leq v(\rho_j) \rho_j \theta_k \leq v(\rho_j) \rho_j  \sum_{k: I_{j+1} \subseteq P_k}   \theta_k
$$
and again by a small modification of the proof of~\cite[Lemma 6.1]{DovettaMarconiSpinolo} this implies 
\be \label{e:mela}
    0 \leq  \mathrm{Tr} [v(\rho_j) \rho_j \theta_k] (\cdot, d^-)   
  \leq  \sum_{k: I_{j+1} \subseteq P_k} \mathrm{Tr}[v (\rho_{j}) \rho_j \theta_k] (\cdot, d^-)
     \stackrel{\eqref{e:barrho}}{=} g(\bar \rho_{j+1}) . 
\eq
We now set 
\be \label{e:thetakb}
    \theta_{kb} : = 
    \left\{ 
      \begin{array}{ll}
                  \displaystyle{
         \frac{ \mathrm{Tr} [v(\rho_j) \rho_j \theta_k] (\cdot, d^-)  }{ g(\bar \rho_{j+1})}
         }
     & \text{if $ g( \bar \rho_{j+1}) \neq 0$} \\
     0 & \text{if $g(\bar \rho_{j+1}) =0$\,.} \\
    \end{array}
    \right.
\eq
Note that, owing to~\eqref{e:mela}, $0 \leq \theta_{kb} \leq 1$. 
To define $\theta_k$ on $I_{j+1}$ we apply Theorem~\ref{t:luglio}
with $]\alpha, \beta[= I_{j+1}$, $b= v(\rho_{j+1})$, $\rho= \rho_{j+1}$ and $\bar \theta= \theta_{kb}$, and we term $\theta_k$ the solution of~\eqref{e:luglio}. By applying the comparison principle, we get that~\eqref{e:zerouno} is satisfied on $I_{j+1}$. The equality~\eqref{e:sum1} is established in {\sc Step 5}. \\
{\sc Step 4:} we show that $\theta_k$ is a solution of~\eqref{e:te} on $I_1 \cup I_2 \cup \cdots \cup I_{j+1}.$ First, we point out that
\be  \label{e:arancia}
    \mathrm{Tr}[v(\rho_{j+1}) \rho_{j+1}] (\cdot, d^+) \stackrel{\eqref{e:f},\eqref{e:maggio}}{=} - g(\rho_{j+1}) (\cdot, d^+)\stackrel{\text{Proposition}~\ref{p:datumclassic}}{=} - g (\bar \rho_{j+1}) . 
\eq
Next, we set $P_{k, j}: = I_1 \cup I_2 \cup \cdots \cup I_{j}$ and $P_{k, j+1}: =
P_{k, j} \cup I_{j+1}$. We recall that by the inductive assumption $\theta_k$ is a solution on $P_{k, j}$, which implies that 
\be \label{e:uva}
    \int_0^T \! \! \int_{P_{k,j}}  \! \! \! 
    r_k \theta_k ( \partial_t \phi + v(r_k) \partial_x \phi) dx dt = \int_0^T \mathrm{Tr}[v(\rho_j) \rho_j \theta_k ] (t, d^-) \phi (t, d) dt, 
    \quad \text{for every $\phi \in C^\infty_c (]0, T[ \times P_{k,j+1})$},
\eq 
where we recall that $r_k$ is obtained by patching together $\rho_1, \dots, \rho_{j}$, see~\eqref{e:patches}. 
On the other hand, since by definition $\theta_k$ is a solution of the initial-boundary value problem on $I_{j+1}$, then
\begin{equation} \label{e:kiwi}
\begin{split}
    \int_0^T \! \! \int_{I_{j+1}}  \! \! \! &
    \rho_{j+1} \theta_k ( \partial_t \phi + v(\rho_{j+1}) \partial_x \phi) dx dt = \int_0^T \mathrm{Tr}[v(\rho_{j+1}) \rho_{j+1}\theta_k ] (t, d^+) \phi (t, d) dt \\ & 
    =  \int_0^T \mathrm{Tr}[v(\rho_{j+i}) \rho_{j+1} ] (t, d^+)  \theta_{kb} 
        \phi (t, d) dt
   \stackrel{\eqref{e:thetakb},\eqref{e:arancia}}{=}
       - \int_0^T
       \mathrm{Tr} [v(\rho_j) \rho_j \theta_k] (t, d^-) 
       \phi (t, d) dt 
\end{split}
\end{equation}
for every $\phi \in C^\infty_c (]0, T[ \times P_{k, j+1})$. 
This implies that 
\begin{equation*} 
\begin{split}
    \int_0^T \! \! \int_{P_{k, j+1}}  \! \! \! &
    r_k \theta_k ( \partial_t \phi + v(r_k) \partial_x \phi) dx dt 
   =  \int_0^T \! \! \int_{P_{k, j}}  \! \! \! 
    r_k \theta_k ( \partial_t \phi + v(r_k) \partial_x \phi) dx dt \\
   & \quad 
+ \int_0^T \! \! \int_{I_{j+1}}  \! \! \!
    \rho_{j+1} \theta_k ( \partial_t \phi + v(\rho_{j+1}) \partial_x \phi) dx dt
     \stackrel{\eqref{e:uva},\eqref{e:kiwi}}{=}0, 
      \quad \text{for every $\phi \in C^\infty_c (]0, T[ \times P_{k,j+1})$},
\end{split}
\end{equation*}
that is $\theta_k$ is a solution of~\eqref{e:te} on $P_{k, j+1} = I_1 \cup I_2 \cup \dots I_{j+1}$. \\
{\sc Step 5:} we establish~\eqref{e:sum1}. Note that at the junction point $d$ between the road $I_j$ and the road $I_{j+1}$ we have 
$$
     \mathrm{Tr}[v(\rho_{j+1}) \rho_{j+1}] (\cdot, d^+)  \stackrel{\eqref{e:arancia}}{=} - g (\bar \rho_{j+1})
    \stackrel{\eqref{e:barrho}}{=} 
     - \sum_{k: I_{j+1} \subseteq P_k} \mathrm{Tr}[v (\rho_{j}) \rho_j \theta_k] (\cdot, d^-)
$$
and owing to Lemma~\ref{l:iff} this yields~\eqref{e:sum1}.
\subsubsection{Uniqueness}
We now establish the uniqueness part in the statement of Theorem~\ref{t:exunisd}. We fix a path $P_k$, term $I_1, \dots, I_{n_k}$ the consecutive roads composing $P_k$ and assume that there are two solutions $\rho_1, \dots, \rho_{n_k}, \theta_k$ and $\rho^\Diamond_1, \dots, \rho^\Diamond_{n_k}, \theta_k^\Diamond$. We want to show that $\rho_1 = \rho_1^\Diamond$ a.e. on $]0, T[ \times I_1$, $\dots, \rho_{n_k} = \rho_{n_k}^\Diamond$ a.e. on $]0, T[ \times I_{n_k}$ and that $\rho_i \theta_k = \rho_i^\Diamond \theta_k^\Diamond$ a.e. on $]0, T[ \times I_i$, for every $i=1, \dots, n_k$. We argue inductively and proceed according to the following steps. \\
{\sc Step 1:} we establish the identities $\rho_1= \rho_1^\Diamond$ and $\theta_k= \theta_k^\Diamond$ on $]0, T[ \times I_1$. Since $\rho_1$ and $\rho_1^\Diamond$ are both entropy admissible solutions of the initial-boundary value problem
~\eqref{e:foglio}, \eqref{e:id1}, \eqref{e:bd1} such that $0 \leq \rho_1, \rho^\Diamond_1 \leq \rho^\ast$, the identity $\rho_1 = \rho_1^\Diamond$ follows from the uniqueness part of Proposition~\ref{p:datainc}. Next, we recall that $\theta_k$ and $\theta_k^\Diamond$ are both solutions of the initial-boundary value problem
$$
   \left\{
   \begin{array}{ll}
     \partial_t [\rho_1 \theta] + \partial_x [v (\rho_1) \rho_1 \theta] =0 \\
    \theta (0, \cdot) = \theta_{k0}, \qquad \theta (\cdot, a) = 
    \bar \theta_k
   \end{array}
  \right.
$$
and hence the identity $\rho_1 \theta_k = \rho_1 \theta_k^\Diamond$
follows from the uniqueness part in Theorem~\ref{t:luglio}. \\
{\sc Step 2:} we assume that $\rho_i = \rho_i^\Diamond$ and $\rho_i \theta_k= \rho_i \theta_k^\Diamond$ on $]0, T[ \times I_i$, for every $i=1, \dots, j$ and we establish the identities $\rho_{j+1} = \rho_{j+1}^\Diamond$ and $\rho_{j+1} \theta_k = \rho_{j+1} \theta_k^\Diamond$ a.e. on $]0, T[ \times I_{j+1}$.  We term $d$ the junction point between $I_j$ and $I_{j+1}$. We recall that by assumption both $\theta_k$ and $\theta_k^\Diamond$ satisfy~\eqref{e:sum1} on $]0, T[ \times I_i$, for every $i=1, \dots, n_k$. We apply Lemma~\ref{l:iff}, recall that $\rho_j = \rho_j^\Diamond$, $\rho_j \theta_k = \rho_j \theta_k^\Diamond$ on $]0, T[ \times I_j$ and from~\eqref{e:fragola} we deduce that 
$
    \mathrm{Tr}[v(\rho_{j+1}) \rho_{j+1}] (\cdot, d^+)
    = 
    \mathrm{Tr}[v(\rho_{j+1}^\Diamond) \rho_{j+1}^\Diamond] (\cdot, d^+)
$
a.e. on $]0, T[$. Owing to the first equality in~\eqref{e:arancia}, this yields the identity $g(\rho_{j+1})(\cdot, d^+) = g(\rho^\Diamond_{j+1})(\cdot, d^+)$ a.e. on $]0, T[$, where $g$ is the same as in~\eqref{e:f} and the traces $g(\rho_{j+1})(\cdot, d^+)$ and $g(\rho^\Diamond_{j+1})(\cdot, d^+)$ are attained in the sense of~\eqref{e:febbraio}. Since by assumption $\rho_{j+1}, \rho^\Diamond_{j+1} \leq \rho^\ast$, then owing to Lemma~\ref{l:lanacaprina}, this implies that $\rho_{j+1} = \rho_{j+1}^\Diamond$ a.e. on $]0, T[ \times I_{j+1}$. Next, we recall that $\rho_j \theta_k= \rho_j \theta_k^\Diamond$ a.e. on $]0, T[\times I_{j}$, we apply~\eqref{e:mirtillo} and recall~\eqref{e:luglio2} and we conclude that $\theta_k$ and $\theta^\Diamond_{k}$ are both solutions 
of the initial-boundary value problem
$$
   \left\{
   \begin{array}{ll}
        \partial_t [\rho_{j+1} \theta] +
       \partial_x [v (\rho_{j+1}) \rho_{j+1} \theta] =0 \\
      \theta (0, \cdot) = \theta_{k0}, \quad 
      \theta(\cdot, d) = \theta_{kb} 
   \end{array}
   \right.
$$
where $\theta_{kb}$ is the same as in~\eqref{e:thetakb} with $\bar \rho_{j+1}$ replaced by $g(\rho_{j+1}) (\cdot, d^+)$. By the uniqueness part in the statement of Theorem~\ref{t:luglio} we conclude that $\rho_{j+1} \theta_k = \rho_{j+1} \theta^\Diamond_k$ a.e. on $]0, T[ \times I_{j+1}$.
\begin{remark}
\label{r:balcone}
Definition~\ref{d:adsol}, Definition~\ref{d:te} and Lemma~\ref{l:iff} extend to the case of more general networks than those considered in the present paper, i.e. networks containing other types of junctions than T-junctions. By arguing as in the proof of Theorem~\ref{t:exunisd} one can show that, if the density functions $\rho_1, \dots, \rho_h$ are assigned, then one can 
construct the functions $\theta_1, \dots, \theta_m$ satisfying Definition~\ref{d:te} and these functions are unique in the sense of~\eqref{e:giugno}. What is missing in the general case is that nothing guarantees that condition~\eqref{e:sum1}, or equivalently the junction condition~\eqref{e:fragola}, is satisfied. 
\end{remark}

\section{Propagation of regularity and stability for the source-destination model} \label{ss:propbv}
\subsection{Proof of Theorem~\ref{t:propbvreg}} 
By the uniqueness part in Theorem~\ref{t:exunisd} it suffices to show that the solution of the distributional source-destination model constructed in \S\ref{sss:exsd}  satisfies Theorem~\ref{t:propbvreg}. In particular, in the proof we show that, under the 
assumptions of Theorem~\ref{t:propbvreg}, $r_k$ is bounded away from $0$ for every $k=1, \dots, m$ and hence the function $\theta_k$ is uniquely determined.   
We fix $k=1, \dots, m$, consider the path $P_k$ and as in \S\ref{sss:exsd} term $I_1, \dots, I_{n_k}$ the consecutive roads composing $P_k$.  \\
{\sc Step 1:} we establish the regularity estimates on $\rho_1$. We recall that $\rho_1$ is obtained by applying Proposition~\ref{p:datainc} and by recalling~\eqref{e:giallo} 
we arrive at 
\be   \label{e:autobus}
     \mathrm{TotVar} \ \rho_1 (t, \cdot) \leq \mathrm{TotVar} \ \bar \rho + \mathrm{TotVar} \ \rho_{10} + |\bar \rho(0^+)- \rho_{10} (a^+)|, \quad \text{for every $t \in ]0, T[$},
\eq 
where we have used Lemma~\ref{l:dafermos} to define the function $\rho_1 (t, \cdot)$ for every $t$. By applying the chain rule for $BV$ functions (see for instance~\cite[Theorem 3.96]{AmbrosioFuscoPallara})  and using equation~\eqref{e:foglio}, we deduce from~\eqref{e:autobus} a control on the total variation of the measure $\partial_t \rho_1$ on $]0, T[ \times ]\alpha, \beta[$ and conclude that $\rho_1 \in BV (]0, T[ \times ]\alpha, \beta[)$. Next, we recall the assumptions on the data and~\eqref{e:maxprincl} and we conclude that
\be \label{e:monopattino} 
     0 < \ee \leq \rho_1 \leq \rho^\ast - \ee, \quad \text{a.e. on $]0, T[ \times I_1$}. 
\eq
We term $d$ the second extremum of $I_1$, we recall that $g$ is given by~\eqref{e:f}, we apply~\eqref{e:treno} and by using the chain rule for $BV$ functions we conclude that $ g(\rho_1)(\cdot, d^-) \in BV (]0, T[)$. Owing to~\eqref{e:monopattino}, this yields 
\be \label{e:nave}
  0 \stackrel{\eqref{e:f}}{<} v(\ee) \ee \leq 
  g(\rho_1)(\cdot, d^-) \leq  v (\rho^\ast - \ee) [\rho^\ast - \ee] .
\eq 
{\sc Step 2:} we establish the regularity estimates for $\theta_k$ on $I_1$. We apply~\cite[Proposition 1.4]{DovettaMarconiSpinolo} with $]\alpha, \beta[= I_1$ and $b = v(\rho_1)$ and we conclude that $\theta_k \in BV (]0, T[ \times I_1)$. Also, owing to~\cite[Theorem 1.5]{DovettaMarconiSpinolo}, there is $\tilde \theta_k \in BV (]0, T[)$ such that 
\be \label{e:aereo}
     \mathrm{Tr} [v(\rho_1) \rho_1\theta_k] (\cdot, d^-) = \tilde \theta_k 
      \mathrm{Tr} [v(\rho_1) \rho_1] (\cdot, d^-)
    \stackrel{\eqref{e:maggio}}{=} 
    \tilde \theta_k 
     g(\rho_1) (\cdot, d^-)\quad 
    \text{a.e. on $]0, T[$}
\eq
and that 
\be \label{e:bicicletta}
    \ee \leq \tilde \theta_k \leq 1 \quad \text{a.e. on $]0, T[$}.
\eq
{\sc Step 3:} we deal with the junction $d$. We have 
\be \label{e:penna}
        g(\rho_2) (\cdot, d^+) \stackrel{\eqref{e:maggio}}{=}
     -  \mathrm{Tr} [v(\rho_2) \rho_2] (\cdot, d^+) \stackrel{\eqref{e:fragola}}{=}
     \sum_{k: I_2 \subseteq P_k}  \mathrm{Tr} [v(\rho_1) \rho_1 \theta_k] (\cdot, d^-)
     \stackrel{\eqref{e:aereo}}{=} g(\rho_1) (\cdot, d^-) \sum_{k: I_2 \subseteq P_k} \tilde \theta_k\,.
\eq 
On the one hand, by \eqref{e:bicicletta} we trivially have $\sum_{k: I_2 \subseteq P_k} \tilde \theta_k\geq\ee$. On the other hand, by linearity $z:= \sum_{k: I_2 \subseteq P_k} \theta_k$ solves
\[
\begin{cases}
\partial_t\left(\rho_1z\right)+\partial_x\left(v(\rho_1)\rho_1z\right)=0 & \text{on }]0,T[\times I_1 \\
z(\cdot,a)=\sum_{k: I_2 \subseteq P_k} \overline\theta_k,  \quad 
z(0,\cdot)=\sum_{k: I_2 \subseteq P_k} \theta_{0k}
\end{cases}
\]
and hence~\cite[Theorem 1.5]{DovettaMarconiSpinolo} yields the existence of $\tilde z\in BV(]0,T[)$ such that
\[
\mathrm{Tr} \left[v(\rho_1)\rho_1\sum_{k: I_2 \subseteq P_k} \theta_k\right](\cdot,d^-)=\tilde z \mathrm{Tr}[v(\rho_1)\rho_1](\cdot,d^-)
\stackrel{\eqref{e:maggio}}{=}
\widetilde z g(\rho_1)(\cdot,d^-)
\qquad\text{a.e. on }]0,T[.
\]
The linearity and the uniqueness of the distributional traces and \eqref{e:nave}  imply
\[
\tilde z =\sum_{k: I_2 \subseteq P_k} \tilde\theta_k\qquad\text{a.e. on }]0,T[
\stackrel{\eqref{e:sommaunoid},~\eqref{e:sommaunobd}}{\implies} \ee \leq 
\tilde z = \sum_{k: I_2 \subseteq P_k} \tilde\theta_k \leq1\qquad\text{a.e. on } ]0,T[ 
\]
and owing to~\eqref{e:nave} and \eqref{e:penna} this yields     
 \be \label{e:metropolitana}    
   \ee^2 v(\ee)
     \leq   g(\rho_2) (\cdot, d^+) \leq v (\rho^\ast - \ee) [\rho^\ast - \ee]. 
\eq
Also,  owing to~\eqref{e:penna}, $g(\rho_2) (\cdot, d^+) \in BV (]0, T[)$. We now recall the construction in {\sc Step 2} of \S\ref{sss:exsd} and in particular that the boundary datum $\bar \rho_2$ for $\rho_2$ is the unique function comprised between $0$ and $\rho^\ast$ such that $g(\bar \rho_2) =  g(\rho_2) (\cdot, d^+)$ a.e. on $]0, T[$. By using~\eqref{e:metropolitana} and the chain rule for BV functions we infer that $\bar \rho_2 \in BV (]0, T[)$. Also, $\tilde \ee \leq \bar \rho_2 \leq \rho^\ast -  \ee$ for a suitable constant $\tilde \ee >0$ which could be explicitely computed if needed. 
Next, we recall the construction in {\sc Step 3} of~\S\ref{sss:exsd} and, by using formula~\eqref{e:thetakb} and recalling that $g(\bar \rho_2)$ is bounded away from $0$, we conclude that the boundary datum for $\theta_k$ at $d$ is 
\be \label{e:cremisi}
   \theta_{kb} = \frac{ \mathrm{Tr} [v(\rho_1) \rho_1] (\cdot, d^-)}{g (\bar \rho_2)}
\eq 
and, owing to~\eqref{e:nave} and~\eqref{e:metropolitana}, this yields $\theta_{kb} \in BV (]0, T[)$ and $\bar \ee \leq \theta_{kb}\leq 1$ for some suitable constant $\bar \ee$ which could be explicitely computed, if needed. We can repeat the argument at {\sc Step 1} and {\sc Step 2} and conclude that $\rho_2 \in BV (]0, T[ \times I_2)$, $\theta_k \in BV (]0, T[ \times I_2)$. \\
{\sc Step 4:} by iterating the argument at the previous steps we conclude that, for every $i=1, \dots, n_k$, $\rho_i \in BV (]0, T[ \times I_i) $ and $\theta_k \in BV (]0, T[ \times I_i)$. To conclude that actually $\theta_k \in BV (]0, T[ \times P_k)$ we apply a ``gluing theorem" for $BV$ functions, see~\cite[Corollary 3.89]{AmbrosioFuscoPallara}. 
\subsection{Proof of Corollary~\ref{c:stability}}
By the uniqueness part of Theorem~\ref{t:exunisd} it suffices to establish the stability of the solution constructed in \S\ref{sss:exsd}. We fix $k=1, \dots, m$ and as in \S\ref{sss:exsd} we term $I_1, \dots, I_{n_k}$ the consecutive roads composing the path $P_k$.  \\
{\sc Step 1:} we show that $\rho^n_1 \to \rho_1$ in $L^1(]0, T[ \times I_1)$. To this end, it suffices to recall the stability of the entropy admissible solutions of initial-boundary value problems with respect to perturbations in the data, see for instance~\cite[Theorem 4.3]{ColomboRossi}. Owing to Lemma~\ref{l:malva} we also have 
\be \label{e:settembre}
     g(\rho^n_1)(\cdot, d^-) \to g (\rho_1) (\cdot, d^-) \; \text{in $L^1 (]0, T[)$},
\eq
where $d$ denotes the second extremum of the interval $I_1$. \\
{\sc Step 2:} we show that $\rho^n_1 \theta^n_k \to \rho_1 \theta_k$ in 
$L^1 (]0, T[ \times I_1)$.  \\
{\sc Step 2A:} we show that there is a sequence $\theta^n_k$ solving the 
initial-boundary value problem
\be \label{e:ottobre}
    \left\{
    \begin{array}{ll}
    \partial_t [\rho^n_1 \theta^n_k ] + 
    \partial_x [v(\rho^n_1) \rho^n_1 \theta^n_k] = 0 \\
    \theta^n_k(0, \cdot) = \theta^n_{k0}, \quad \theta^n_k ( \cdot, a) = \bar \theta_k^n
    \end{array}
   \right.
\eq
on $]0, T[\times I_1$ such that $\theta^n_k \weaks \theta_k$ weakly$^\ast$ in $L^\infty(]0, T[ \times I_1)$, where $\theta_k$ is a solution of the initial-boundary value problem~\eqref{e:te},\eqref{e:id2},\eqref{e:bd2}. To this end, we recall that $\theta^n_k$ is constructed in  \S\ref{sss:exsd} by applying~\cite[Theorem 1.2]{DovettaMarconiSpinolo}, and this yields a $L^\infty$ bound on $\theta_k^n$ in terms of $\| \theta^n_{k 0} \|_{L^\infty}$ and $\| \bar \theta^n_k \|_{L^\infty}$. Owing to~\eqref{e:sommaunoid} and~\eqref{e:sommaunobd} we conclude that 
$\| \theta^n_k \|_{L^\infty}$ is uniformly bounded and hence weakly$^\ast$ converges (up to subsequences) to some limit function $\theta_k$. By using {\sc Step 1} we can pass to the limit in the distributional formulation of~\eqref{e:ottobre} and conclude that 
$\theta_k$ is a solution of~\eqref{e:te},~\eqref{e:id2},~\eqref{e:bd2}. \\
{\sc Step 2B:} we show that $\rho_1 (\theta^n_k)^2 \weaks \rho_1 (\theta_k)^2$ 
weakly$^\ast$ in $L^\infty (]0, T[ \times I_1)$. Owing to the proof of~\cite[Proposition 3.11]{DovettaMarconiSpinolo}, $(\theta^n_k)^2$ is a solution of the initial-boundary value problem~\eqref{e:ottobre} with $\theta^n_{0k}$ and $\bar \theta_k^n$ replaced by $(\theta^n_{0k})^2$ and $(\bar \theta_k^n)^2$, respectively.  Also, $\| (\theta_k^n)^2\|_{L^\infty}$ is uniformly bounded because so is $\| \theta_k^n \|_{L^\infty}$ and hence, up to subsequences, $(\theta_k^n)^2$ weakly$^\ast$ converges in $L^\infty (]0, T[ \times I_1)$ to some limit function $\gamma$. Since $\rho^n_1 \to \rho_1$ strongly in $L^1 (]0, T[ \times I_1)$ by {\sc Step 1}, this implies that $\rho^n_1 (\theta_k^n)^2 \weaks \rho_1 \gamma$  weakly$^\ast$ in $L^\infty (]0, T[ \times I_1)$. By passing to the limit in the distributional formulation we get that $\rho_1 \gamma$ is a solution of~\eqref{e:ottobre} with $\theta^n_{0k}$ and $\bar \theta_k^n$ replaced by $(\theta_{0k})^2$ and $(\bar \theta_k)^2$, respectively. Since by~\cite[Proposition 3.11]{DovettaMarconiSpinolo} $\rho_1 (\theta_k)^2$ is a solution of the same initial-boundary value problem, then by the uniqueness part of~\cite[Theorem 1.2]{DovettaMarconiSpinolo} we have $\rho_1 \gamma = \rho_1 (\theta_k)^2$. \\
{\sc Step 2C:} we conclude the proof of {\sc Step 2}.  We recall that $\rho^n_1 \to \rho_1$ strongly in $L^1 (]0, T[ \times I_1)$ owing to {\sc Step 1}: by {\sc Step 2A}, this implies that $\rho_1 \theta^n_k \weaks \rho_1 \theta_k$ weakly$^\ast$ in $L^\infty (]0, T[ \times I_1)$ and henceforth weakly in $L^2 (]0, T[ \times I_1)$.  By {\sc Step 2B}, it also implies that $(\rho_1 \theta^n_k)^2 \weaks (\rho_1 \theta_k)^2$ weakly$^\ast$ in $L^\infty (]0, T[ \times I_1)$ and henceforth weakly in $L^2 (]0, T[ \times I_1)$. We conclude that  $\rho_1 \theta^n_k \to \rho_1 \theta_k$ strongly in $L^2 (]0, T[ \times I_1)$ and henceforth strongly in $L^1 (]0, T[ \times I_1)$. \\
{\sc Step 3:} we show that $\bar \rho^n_2 \to \bar \rho_2$ strongly in $L^1 (]0, T[)$, where $\bar \rho^n_2$ and $\bar \rho_2$ are the boundary data for $\rho_2^n$ and $\rho_2$, respectively. We recall that, by the construction in~\S\ref{sss:exsd}, $\bar \rho_2$ is the function confined between $0$ and $\rho^\ast$ such that 
$$
     g(\bar \rho_2) = \sum_{k: I_2 \subseteq P_k} \mathrm{Tr}[v (\rho_1) \rho_1 \theta_k] (\cdot, d^-),
$$
where $d$ is the second extremum of $I_1$. Hence, to establish the convergence 
$\bar \rho^n_2 \to \bar \rho_2$ it suffices to show that $ \mathrm{Tr}[v (\rho^n_1) \rho^n_1 \theta^n_k] (\cdot, d^-) \to  \mathrm{Tr}[v (\rho_1) \rho_1 \theta_k] (\cdot, d^-)$ strongly in $L^1 (]0, T[)$ for every $k=1, \dots, m$. \\
{\sc Step 3A:} we show that  $ \mathrm{Tr}[v (\rho^n_1) \rho^n_1 \theta^n_k] (\cdot, d^-) \weaks  \mathrm{Tr}[v (\rho_1) \rho_1 \theta_k] (\cdot, d^-)$ weakly$^\ast$ in $L^\infty (]0, T[)$. Owing to~\cite[Proposition 3.2]{AmbrosioCrippaManiglia} and to the proof of~\cite[Lemma 3.3]{CrippaDonadelloSpinolo}, $\|  \mathrm{Tr}[v (\rho^n_1) \rho^n_1 \theta^n_k] \|_{L^\infty}$ is uniformly bounded in terms of $\| \rho^n_1 \|_{L^\infty}$ and $\| \theta_k^n \|_{L^\infty}$ and hence up to subsequences converges weakly$^\ast$ in $L^\infty (]0, T[)$ to  some function $\delta$. By recalling {\sc Step 1} and {\sc Step 2} and passing to the limit in the definition of distributional trace we get that 
$\delta= \mathrm{Tr}[v (\rho_1) \rho_1 \theta_k]$.  \\
{\sc Step 3B:} by recalling {\sc Step 2B} we can repeat the same argument as in {\sc Step 3A} and conclude that $ \mathrm{Tr}[v (\rho^n_1) \rho^n_1 (\theta^n_k)^2] (\cdot, d^-) \weaks  \mathrm{Tr}[v (\rho_1) \rho_1 (\theta_k)^2] (\cdot, d^-)$ weakly$^\ast$ in $L^\infty (]0, T[)$. \\
{\sc Step 3C:} owing to the trace renormalization property given by~\cite[Theorem 4.2]{DovettaMarconiSpinolo} we have 
\begin{equation*}
\begin{split}
    \big(  \mathrm{Tr}[v (\rho^n_1) \rho^n_1 \theta^n_k] (\cdot, d^-)  \big)^2 
   &  \stackrel{\text{\cite[Theorem 4.2]{DovettaMarconiSpinolo}}}{=}  
    \mathrm{Tr}[v (\rho^n_1) \rho^n_1 (\theta^n_k)^2](\cdot, d^-)  \mathrm{Tr}[v (\rho^n_1) \rho^n_1] (\cdot, d^-) \\ &
    \stackrel{\eqref{e:maggio}}{=}
      \mathrm{Tr}[v (\rho^n_1) \rho^n_1 (\theta^n_k)^2] (\cdot, d^-)  
   g (\rho^n_1) (\cdot, d^-) .
\end{split}
\end{equation*}
Owing to {\sc Step 3B} and~\eqref{e:settembre} this yields that 
$\big(  \mathrm{Tr}[v (\rho^n_1) \rho^n_1 \theta^n_k] (\cdot, d^-)  \big)^2$ weakly$^\ast$ converges in $L^\infty (]0, T[)$ to $\big(  \mathrm{Tr}[v (\rho^n_1) \rho^n_1 \theta^n_k] (\cdot, d^-)  \big)^2$ and by recalling {\sc Step 3A} and repeating the same argument as in {\sc Step 2C}  it implies that $ \mathrm{Tr}[v (\rho^n_1) \rho^n_1 \theta^n_k] (\cdot, d^-)$  strongly converges in $L^1 (]0, T[)$ to $  \mathrm{Tr}[v (\rho_1) \rho_1 \theta_k] (\cdot, d^-)$. \\ 
{\sc Step 4:} owing to {\sc Step 3}, we can repeat the same argument as in {\sc Step 1} and conclude that $\rho_2^n \to \rho_2$ in $L^1(]0, T[ \times I_2)$. Next, we recall that, by the analysis in \S\ref{sss:exsd},  $\theta_k$ is 
defined on $]0, T[ \times I_2$ by solving an initial-boundary value problem analogous to~\eqref{e:thetakb} and with boundary datum $\theta_{kb}$ given by~\eqref{e:thetakb}. We now want to show that $\rho^n_2 \theta_k^n \to \rho_2 \theta_k$. Note that to repeat the same argument as in {\sc Step 2} it suffices to show that 
\be \label{e:dicembre}
     \theta^n_{kb} g(\bar \rho^n_2)  \to 
      \theta_{kb} g(\bar \rho_2) , \quad 
     (\theta^n_{kb})^2 g(\bar \rho^n_2)  \to 
      (\theta_{kb})^2 g(\bar \rho_2) , \; 
    \text{in $L^1 (]0, T[)$}. 
\eq
{\sc Step 4A:} we establish the first convergence result in~\eqref{e:dicembre}. It suffices to recall that, owing to~\eqref{e:thetakb}, $\theta_{kb} g(\bar \rho_2)  = \mathrm{Tr} [v(\rho_1) \rho_1 \theta_k] (\cdot, d^-)$ and then recall {\sc Step 3C}. \\
{\sc Step 4B:} we establish the second convergence result in~\eqref{e:dicembre}. 
We want to apply the Lebesgue Dominated Convergence Theorem. First, we recall that $|\theta^n_{kb}| \leq 1$ owing to~\eqref{e:mela} and we conclude that 
$\|  (\theta^n_{kb})^2 g(\rho^n_2) (\cdot, d^+) \|_{L^\infty}$ is uniformly bounded. We are left to establish the a.e.  pointwise convergence. First, we recall {\sc Step 3C} and conclude that, up to subsequences, $ \mathrm{Tr}[v (\rho^n_1) \rho^n_1 \theta^n_k] (t, d^-)$ converges to $ \mathrm{Tr}[v (\rho_1) \rho_1 \theta_k] (t, d^-)$  and $g(\bar \rho^n_2 (t))$ converges to $g(\bar \rho_2 (t))$ for a.e. $t \in ]0, T[$. We fix a $t \in ]0, T[$ such that the above convergence results hold true and we distinguish between two cases. If $g(\bar \rho_2 (t)) \neq 0$, then for $n$ sufficiently large
$$
     (\theta^n_{kb}(t) )^2 g(\bar \rho^n_2 (t)) = \frac{\Big(\mathrm{Tr}[v (\rho^n_1) \rho^n_1 \theta^n_k] (t, d^-)\Big)^2}{ g(\bar \rho^n_2 (t))}  \to 
      \frac{\Big( \mathrm{Tr}[v (\rho_1) \rho_1 \theta_k ] (t, d^-)\Big)^2}{ g(\bar \rho_2 (t))} 
     =  (\theta_{kb}(t) )^2 g(\bar \rho_2 (t))
     \quad 
    \text{as $n \to + \infty$.}
$$
If $g(\bar \rho_2 (t)) = 0$ we argue as follows: since $|\theta^n_{kb}(t)|\leq 1$, then $|(\theta^n_{kb}(t) )^2 g(\bar \rho^n_2 (t))| \leq g(\bar \rho^n_2 (t))$
and hence it converges to $0$ as $n \to + \infty$. This concludes the proof of the a.e. pointwise convergence and hence of {\sc Step 4}. \\
{\sc Step 5:} by iterating the argument at the previous steps we establish the desired stability result.   

\section*{Acknowledgments}
The authors wish to thank Maya Briani and Mauro Garavello for several interesting discussions. S.D. and L.V.S. are partially supported by the INDAM-GNAMPA project 2020 \emph{Modelli differenziali alle derivate parziali per fenomeni di interazione.} E.M. is supported by the SNF Grant 182565. Part
of this work was done while S.D. was affiliated to IMATI-CNR, Pavia.
\bibliographystyle{plain}
\bibliography{sd}

\begin{thebibliography}{10}

\bibitem{AmbrosioCrippaManiglia}
L.~Ambrosio, G.~Crippa, and S.~Maniglia.
\newblock Traces and fine properties of a {$BD$} class of vector fields and
  applications.
\newblock {\em Ann. Fac. Sci. Toulouse Math. (6)}, 14(4):527--561, 2005.

\bibitem{AmbrosioDL}
L.~{Ambrosio} and C.~{De Lellis}.
\newblock {A note on admissible solutions of 1D scalar conservation laws and 2D
  Hamilton-Jacobi equations}.
\newblock {\em {J. Hyperbolic Differ. Equ.}}, 1(4):813--826, 2004.

\bibitem{AmbrosioFuscoPallara}
L.~Ambrosio, N.~Fusco, and D.~Pallara.
\newblock {\em Functions of bounded variation and free discontinuity problems}.
\newblock Oxford Mathematical Monographs. The Clarendon Press, Oxford
  University Press, New York, 2000.

\bibitem{Anzellotti}
G.~{Anzellotti}.
\newblock {Pairings between measures and bounded functions and compensated
  compactness}.
\newblock {\em {Ann. Mat. Pura Appl. (4)}}, 135:293--318, 1983.

\bibitem{BLN79}
C.~{Bardos}, A.-Y. {Le Roux}, and J.~C. {Nedelec}.
\newblock {First order quasilinear equations with boundary conditions}.
\newblock {\em {Commun. Partial Differ. Equations}}, 4:1017--1034, 1979.

\bibitem{BellomoDogbe}
N.~{Bellomo} and C.~{Dogbe}.
\newblock {On the modeling of traffic and crowds: a survey of models,
  speculations, and perspectives}.
\newblock {\em {SIAM Rev.}}, 53(3):409--463, 2011.

\bibitem{BianchiniMarconi}
S.~{Bianchini} and E.~{Marconi}.
\newblock {On the structure of \({L^\infty}\)-entropy solutions to scalar
  conservation laws in one-space dimension}.
\newblock {\em {Arch. Ration. Mech. Anal.}}, 226(1):441--493, 2017.

\bibitem{BGJ}
C.~{Bourdarias}, M.~{Gisclon}, and S.~{Junca}.
\newblock {Fractional \(BV\) spaces and applications to scalar conservation
  laws}.
\newblock {\em {J. Hyperbolic Differ. Equ.}}, 11(4):655--677, 2014.

\bibitem{Bressan_book}
A.~{Bressan}.
\newblock {\em {Hyperbolic systems of conservation laws. The one-dimensional
  Cauchy problem}}, volume~20.
\newblock Oxford: Oxford University Press, 2000.

\bibitem{BressanCanicGaravelloHertyPiccoli}
A.~{Bressan}, S.~{\v{C}ani\'c}, M.~{Garavello}, M.~{Herty}, and B.~{Piccoli}.
\newblock {Flows on networks: recent results and perspectives}.
\newblock {\em {EMS Surv. Math. Sci.}}, 1(1):47--111, 2014.

\bibitem{BressanNguyen}
A.~{Bressan} and K.~T. {Nguyen}.
\newblock {Conservation law models for traffic flow on a network of roads}.
\newblock {\em {Netw. Heterog. Media}}, 10(2):255--293, 2015.

\bibitem{BressanYu}
A.~{Bressan} and F.~{Yu}.
\newblock {Continuous Riemann solvers for traffic flow at a junction}.
\newblock {\em {Discrete Contin. Dyn. Syst.}}, 35(9):4149--4171, 2015.

\bibitem{BrianiCristiani}
M.~{Briani} and E.~{Cristiani}.
\newblock {An easy-to-use algorithm for simulating traffic flow on networks:
  theoretical study}.
\newblock {\em {Netw. Heterog. Media}}, 9(3):519--552, 2014.

\bibitem{ChenZiemerTorres}
G.-Q. {Chen}, W.~P. {Ziemer}, and M.~{Torres}.
\newblock {Gauss-Green theorem for weakly differentiable vector fields, sets of
  finite perimeter, and balance laws}.
\newblock {\em {Commun. Pure Appl. Math.}}, 62(2):242--304, 2009.

\bibitem{CocliteGaravelloPiccoli}
G.~M. {Coclite}, M.~{Garavello}, and B.~{Piccoli}.
\newblock {Traffic flow on a road network}.
\newblock {\em {SIAM J. Math. Anal.}}, 36(6):1862--1886, 2005.

\bibitem{ColomboRossi}
R.~M. {Colombo} and E.~{Rossi}.
\newblock {Rigorous estimates on balance laws in bounded domains}.
\newblock {\em {Acta Math. Sci., Ser. B, Engl. Ed.}}, 35(4):906--944, 2015.

\bibitem{CrippaDonadelloSpinolo}
G.~{Crippa}, C.~{Donadello}, and L.~V. {Spinolo}.
\newblock {Initial-boundary value problems for continuity equations with BV
  coefficients}.
\newblock {\em {J. Math. Pures Appl. (9)}}, 102(1):79--98, 2014.

\bibitem{CrippaOttoW}
G.~{Crippa}, F.~{Otto}, and M.~{Westdickenberg}.
\newblock {Regularizing effect of nonlinearity in multidimensional scalar
  conservation laws}.
\newblock In {\em {Transport equations and multi-D hyperbolic conservation
  laws}}, pages 77--128. Berlin: Springer, 2008.

\bibitem{Daf}
C.~M. Dafermos.
\newblock Regularity and large time behaviour of solutions of a conservation
  law without convexity.
\newblock {\em Proc. Roy. Soc. Edinburgh Sect. A}, 99(3-4):201--239, 1985.

\bibitem{Dafermos:book}
C.~M. Dafermos.
\newblock {\em Hyperbolic conservation laws in continuum physics}, volume 325
  of {\em Grundlehren der Mathematischen Wissenschaften [Fundamental Principles
  of Mathematical Sciences]}.
\newblock Springer-Verlag, Berlin, fourth edition, 2016.

\bibitem{DL07}
C.~De~Lellis.
\newblock Notes on hyperbolic systems of conservation laws and transport
  equations.
\newblock {\em Handbook of Differential Equations: Evolutionary Equations},
  3:277--382, 2007.

\bibitem{DovettaMarconiSpinolo}
S.~{Dovetta}, E.~{Marconi}, and L.V. {Spinolo}.
\newblock {Initial-boundary value problems for merely bounded nearly
  incompressible vector fields in one space dimension}.
\newblock {\em {Preprint ArXiv:2105.11157}}, 2021.

\bibitem{GaravelloHanPiccoli}
M.~Garavello, K.~Han, and B.~Piccoli.
\newblock {\em Models for Vehicular Traffic on Networks}.
\newblock American Institute of Mathematical Sciences, 2016.

\bibitem{GaravelloMarcellini}
M.~{Garavello} and F.~{Marcellini}.
\newblock {Global weak solutions to the Cauchy problem for a two-phase model at
  a node}.
\newblock {\em {SIAM J. Math. Anal.}}, 52(2):1567--1590, 2020.

\bibitem{GaravelloPiccoli:CMS}
M.~{Garavello} and B.~{Piccoli}.
\newblock {Source-destination flow on a road network}.
\newblock {\em {Commun. Math. Sci.}}, 3(3):261--283, 2005.

\bibitem{GaravelloPiccoli:AIHP}
M.~{Garavello} and B.~{Piccoli}.
\newblock {Conservation laws on complex networks}.
\newblock {\em {Ann. Inst. Henri Poincar\'e, Anal. Non Lin\'eaire}},
  26(5):1925--1951, 2009.

\bibitem{HW}
M.~Hilliges and W.~Weidlich.
\newblock A phenomenological model for dynamic traffic flow in networks.
\newblock {\em Transportation Research Part B.}, 29:407--431, 1995.

\bibitem{HoldenRisebro}
H.~{Holden} and N.~H. {Risebro}.
\newblock {A mathematical model of traffic flow on a network of unidirectional
  roads}.
\newblock {\em {SIAM J. Math. Anal.}}, 26(4):999--1017, 1995.

\bibitem{Jabin}
P.-E. {Jabin}.
\newblock {Some regularizing methods for transport equations and the regularity
  of solutions to scalar conservation laws}.
\newblock {\em {S\'emin. \'Equ. D\'eriv. Partielles, \'Ec. Polytech., Cent.
  Math. Laurent Schwartz, Palaiseau}}, 2008-2009:ex, 2010.

\bibitem{Kruzkov}
S.~N. Kru{\v{z}}kov.
\newblock First order quasilinear equations with several independent variables.
\newblock {\em Mat. Sb. (N.S.)}, 81 (123):228--255, 1970.

\bibitem{KwonVasseur}
Y.-S. {Kwon} and A.~{Vasseur}.
\newblock {Strong traces for solutions to scalar conservation laws with general
  flux}.
\newblock {\em {Arch. Ration. Mech. Anal.}}, 185(3):495--513, 2007.

\bibitem{LW}
M.~Lighthill and G.~Whitham.
\newblock On kinematic waves. {II. A} theory of traffic flow on long crowded
  roads.
\newblock {\em Proceedings of the Royal Society of London: Series A.},
  229:317--345, 1955.

\bibitem{Marconi}
E.~{Marconi}.
\newblock {Regularity estimates for scalar conservation laws in one space
  dimension}.
\newblock {\em {J. Hyperbolic Differ. Equ.}}, 15(4):623--691, 2018.

\bibitem{Oleinik}
O.~A. Ole{\u\i}nik.
\newblock Discontinuous solutions of non-linear differential equations.
\newblock {\em Uspehi Mat. Nauk (N.S.)}, 12(3(75)):3--73, 1957.

\bibitem{Otto}
F.~{Otto}.
\newblock {Initial-boundary value problem for a scalar conservation law}.
\newblock {\em {C. R. Acad. Sci., Paris, S\'er. I}}, 322(8):729--734, 1996.

\bibitem{Panov:traces}
E.~Yu. {Panov}.
\newblock {Existence of strong traces for quasi-solutions of multidimensional
  conservation laws}.
\newblock {\em {J. Hyperbolic Differ. Equ.}}, 4(4):729--770, 2007.

\bibitem{R}
P.~I. Richards.
\newblock Shock waves on the highway.
\newblock {\em Operations Res.}, 4:42--51, 1956.

\bibitem{Rossi}
E.~{Rossi}.
\newblock {Definitions of solutions to the IBVP for multi-dimensional scalar
  balance laws}.
\newblock {\em {J. Hyperbolic Differ. Equ.}}, 15(2):349--374, 2018.

\bibitem{Schaeffer}
D.~G. Schaeffer.
\newblock A regularity theorem for conservation laws.
\newblock {\em Advances in Math.}, 11:368--386, 1973.

\bibitem{Serre}
D.~{Serre}.
\newblock {\em {Systems of conservation laws 2. Geometric structures,
  oscillations, and initial-boundary value problems. Translated from the French
  by I. N. Sneddon}}.
\newblock Cambridge: Cambridge University Press, 2000.

\bibitem{Vasseur}
A.~{Vasseur}.
\newblock {Strong traces for solutions of multidimensional scalar conservation
  laws}.
\newblock {\em {Arch. Ration. Mech. Anal.}}, 160(3):181--193, 2001.

\end{thebibliography}
\end{document}